\documentclass[12pt]{amsart}
\usepackage{amsthm,amssymb,amsmath,amscd,epic,eepic}
\usepackage{amsthm}
\usepackage[all]{xy}
\usepackage{graphicx}
\theoremstyle{plain}

\newcommand{\mute}[2] {}
\newcommand     {\printname}[1] {}
\newcommand{\labell}[1] {\label{#1}}

\newtheorem{theorem}{Theorem}[section]
\newtheorem*{theorem-non}{Theorem}
\newtheorem{lemma}[theorem]{Lemma}
\newtheorem{proposition}[theorem]{Proposition}

\newtheorem{corollary}[theorem]{Corollary}

\newtheorem{remark}[theorem]{Remark}
\newtheorem{definition}[theorem]{Definition}
\newtheorem{example}[theorem]{Example}

\def    \calO             {{\mathcal O}}
\def \fk {{\mathfrak k}}

\def \af {{l}}
\def    \conv  {\operatorname{conv}}

\def    \inv    {^{-1}}
\def    \CP	{{\mathbb C}{\mathbb P}}
\def    \R	{{\mathbb R}}
\def    \P	{{\mathbb P}}

\def    \Q	{{\mathbb Q}}
\def    \C	{{\mathbb C}}
\def    \Z       {{\mathbb Z}}

\def    \F   {{\mathcal F}}
\def    \n	{{[n]}}
\def    \m	{{[m]}}
\def    \l	{{[\ell]}}
\def    \k	{{[k]}}
\def    \j	{{[j]}}

\def   \mat {M_n(\mathbb{Z})}

\def    \tLambda	{{\widetilde \Lambda}}
\def    \tXi	{{\widetilde \Xi}}
\def    \tlambda	{{\widetilde \lambda}}
\def    \tomega{{\widetilde \omega}}

\def    \tA	{{\widetilde A}}

\def    \ts	{{\widetilde s}}

\def    \tl	{{\widetilde l}}
\def    \ti	{{\widetilde i}}
\def    \tm	{{\widetilde m}}
\def    \ts	{{\widetilde s}}
\def    \tx	{{\widetilde x}}

\def    \tz	{{\widetilde z}}
\def    \tM	{{\widetilde M}}

\def    \teta	{{\widetilde \eta}}
\def    \calL   {{\mathcal L}}
\DeclareMathOperator{\hh}{H}

\def	\omegaFS {{\omega_{\operatorname{FS}}}}

\def    \Cone      {{\operatorname{Cone}}}
\def    \smooth      {{\operatorname{smooth}}}

\def    \Proj      {{\operatorname{Proj}}\, }

\begin{document}

\title[Symplectic cohomological rigidity]
{Symplectic cohomological rigidity via toric degnerations}

\author[Milena Pabiniak]{Milena Pabiniak}
\address{Mathematisches Institut, Universit\"at zu K\"oln, Weyertal 86-90, D-50931 K\"oln, Germany}
\email{pabiniak@math.uni-koeln.de}

\author[Susan Tolman]{Susan Tolman}
\address{Department of Mathematics, University of Illinois at Urbana-Champaign,
 Urbana, IL 61801}
\email{stolman@math.uiuc.edu}


\begin{abstract}
In this paper we study whether symplectic toric manifolds are symplectically cohomologically rigid.
Here we say that symplectic cohomological rigidity holds for some family of symplectic manifolds if the members of that family can be distinguished by their integral cohomology rings and the cohomology classes of their symplectic forms.
We show how toric degenerations can be used to produce the symplectomorphisms necessary to answer this question. 
As a consequence we prove that symplectic cohomological rigidity holds for the family of symplectic Bott manifolds with rational symplectic form whose rational cohomology ring is isomorphic to $\hh^*((\CP^1)^n;\Q)$ for some $n$.
In particular, we classify such manifolds up to symplectomorphism.
Moreover, we prove that any symplectic toric manifold with rational symplectic form whose integral cohomology ring is isomorphic to $\hh^*((\CP^1)^n;\Z)$  is symplectomorphic to $(\CP^1)^n$  with a product symplectic structure.
\end{abstract}
\maketitle

\section{Introduction}

One powerful method to study (smooth) manifolds is to calculate their invariants.
In particular,  if two manifolds are diffeomorphic, then their integral cohomology rings are isomorphic. 
We say that a family of manifolds is {\bf cohomologically rigid} if the converse holds,
that is,  if any two manifolds in the family with isomorphic integral cohomology rings are  diffeomorphic\footnote{In the literature,
the term ``cohomologically rigid" is sometimes defined in terms
of  {\em homeomorphism}, rather than {\em diffeomorphism}.}.

Most families of manifolds cannot be classified by their cohomology ring in this way, but their
are some important exceptions.
For example, cohomological rigidity holds for $2$ dimensional manifolds.
More interestingly, Freedman proved that  closed simply connected $4$ manifolds
are classified up to homeomorphism by their cohomology ring.
In contrast, there are  infinite families of closed simply connected
$4$ manifolds that are homeomorphic (and so have same cohomology ring) but not diffeomorphic.

Masuda and Suh asked if  cohomological rigidity holds for another important family: smooth toric varieties. 
Here, by {\bf toric variety} we mean a compact algebraic variety $X$ of complex dimension $n$ equipped with an action of the algebraic torus $(\C^\times)^n$ with a dense orbit.  Smooth toric
varieties are classified by, and can be studied in terms of, a purely combinatorial invariant, called a {\bf fan}.
As Masuda shows, this implies that smooth toric varieties  are equivariantly diffeomorphic exactly if their integral equivariant cohomology rings are isomorphic as algebras over the equivariant cohomology of a point \cite{M}.
In contrast, the classification of smooth toric varieties up to diffeomorphism is still poorly understood; in particular, Masuda and Suh's question is still open. 
Nevertheless, there have been some significant partial positive results. 

To state these results, we need several definitions. First, a  
{\bf Bott manifold} is the total space of a sequence of $\CP^1$ bundles
starting with $\CP^1$, where each $\CP^1$ bundle is the projectivization
of the Whitney sum of two holomorphic line bundles \cite{GK}.
Bott manifolds are smooth toric varieties: at each stage the torus action on the base lifts to the line bundles,
and hence to their projectivizations; moreover, there is an action of the multiplicative group $\C^\times$ on the total space that restricts
on each fiber to the standard action on $\CP^1$.
Second, we say that  {\bf strong cohomological rigidity} holds for a family of manifolds if every isomorphism between the
cohomology rings of any two manifolds in the family is induced by a diffeomorphism.

Choi and Masuda proved that strong cohomological rigidity holds for the family of Bott manifolds whose rational cohomology ring is isomorphic to that of a product of projective planes  \cite{CM}.
This reproves an earlier result of Masuda and Panov stating  that a Bott manifold $X$ is 
diffeomorphic to $\prod \CP^1$ if and only if $\hh^*(X;\Z)=\hh^*\big( \prod \CP^{1};\Z\big)$ \cite[Theorem 5.1]{MP}.
In \cite{CMSgenBott}, Choi, Masuda, and Suh considered a generalization of Bott manifolds where the
 base and fibers can be projective spaces of any dimension.
 They proved that such a manifold is diffeomorphic to $\prod_{j}\CP^{n_j}$ if and only if its integral cohomology ring is isomorphic to $\hh^*\big(\prod_{j}\CP^{n_j};\Z\big)$. 
Additionally,  strong cohomological rigidity holds for Bott manifolds of dimension at most $8$ \cite{C}, and for the family Bott manifolds for which at most one bundle in the fibration is not trivial \cite{CS}. For a more detailed survey, see \cite{CMSsurvey}.

In this paper, we study the symplectic counterpart of
the question posed by Masuda and Suh.
We  say that a family of symplectic manifolds is {\bf symplectically cohomologically rigid} if two 
symplectic manifolds $(M,\omega)$ and $(\tM,\tomega)$ in the family are symplectomorphic
exactly if  there  is isomorphism $\hh^*(M;\Z) \rightarrow \hh^*(\tM;\Z)$ sending the class $[\omega]$ to the class $[\tomega]$.
Our main goal is to determine if the family of symplectic toric manifolds is symplectically cohomologically rigid. 
Here, by  {\bf symplectic toric manifold} we mean a  $2n$-dimensional closed connected symplectic manifold $(M,\omega)$ equipped with an effective $(S^1)^n$ action with {\bf moment map} $\mu \colon M \to \R^n$, i.e, 
$\iota_{\xi_i} \omega = - d \mu_i$ for each $i \in \{1,\dots,n\}$, where $\xi_i$ is the vector
field induced on $M$ by the $i$'th copy of $(S^1)$.
The moment polytope $\Delta := \mu(M)$ of $M$
is a {\bf smooth}  polytope, i.e.,
 $\Delta$ is a rational polytope and the primitive integral vectors perpendicular to the facets meeting at
 at any one vertex of the {\bf moment polytope} $\Delta$ form a $\Z$ basis for $\Z^n$.
As  Delzant proved, these polytopes classify symplectic toric manifolds:
If two symplectic toric manifolds have the same moment polytope,
then there exists an equivariant symplectomorphism between them that intertwines the moment maps.
Moreover, every smooth polytope arises as the moment polytope of an $(S^1)^n \subset (\C^\times)^n$ invariant K\"ahler form on
a smooth toric variety.
As a  consequence,  every symplectic toric manifold is equivariantly diffeomorphic to a smooth toric variety.

In contrast, not much is known about the classification of symplectic toric manifolds up
to symplectomorphism.
Karshon, Kessler and Pinsonnault proved that any $4$-dimensional symplectic toric  manifold admits only finitely many inequivalent toric action, that is, it is only symplectomorphic to a finite number of distinct symplectic toric manifolds \cite{KKP}. 
This result was later generalized to all dimensions by McDuff  \cite{McD}.
Even more relevantly for our paper,  she also  proved that the family of
symplectic toric manifolds whose cohomology ring is isomorphic to that of a product of projective spaces is symplectically cohomologically rigid  \cite[Section 2.4]{McD}.

In this work, we prove that strong symplectic cohomological rigidity  holds for the family of symplectic Bott manifolds
whose rational cohomology ring is isomorphic to that of a product of $\CP^1$'s. 
More precisely, we prove the theorem below, which also appears in Section~\ref{sec proof} as Theorem~\ref{main theorem}.
As we explain in Section~\ref{sec Bott manifolds}, 
a symplectic Bott manifold is a Bott manifold  equipped with an $(S^1)^n$ invariant K\"ahler form.
Our main tool is a new method of constructing symplectomorphism between
symplectic toric manifolds using toric degeneration, which is described   in more detail below.

\begin{theorem-non}
Let $(M,\omega)$ and $(\tM, \tomega)$ be symplectic  Bott manifolds with rational symplectic forms\footnote{The result still holds if
some multiples of the symplectic forms are  rational. } 
such that $$\hh^*(M;\Q) \simeq \hh^*(\tM;\Q)  \simeq \hh^*( (\CP^1)^n;\Q).$$ 
Given a ring isomorphism $F \colon \hh^*(M;\Z) \rightarrow \hh^*(\tM;\Z)$ 
such that \footnote{By a slight abuse of notation,  we also denote the induced map from
$\hh^*(M;\R)$ to $ \hh^*(\tM;\R)$ by $F$.}
$F([\omega]) = [\tomega]$, there exists a symplectomorphism $f$
from $(\tM,\tomega)$ to $(M,\omega)$ so that the induced map $f^*\colon \hh^*(M;\Z) \rightarrow \hh^*(\tM;\Z)$  is equal to $F$.
\end{theorem-non}

Given $\lambda \in \R^n$ such that $\lambda_i \neq 0$ for all $i$,
let $\omega_\lambda \in \Omega^2((\CP^1)^n)$ be the symplectic form $\sum_i \lambda_i \pi_i^*(\omegaFS)$, where $\pi_i \colon (\CP^1)^n \to \CP^1$ denotes the projection onto the $i$-th factor,
and $\omegaFS$ is the Fubini-Study form on $\CP^1$ with area $1$.

\begin{corollary}\label{z trivial}
Let $(M,\omega)$  be a symplectic toric manifold with rational symplectic form, 
and let $F \colon \hh^*(M;\Z) \to \hh^*((\CP^1)^n;\Z)$ be an isomorphism.
Then there exists $\lambda \in \Q^n$
with $\lambda_i \neq 0$ for all $i$ and a symplectomorphism $f$ from 
$((\CP^1)^n, \omega_\lambda)$
to
$(M,\omega)$  
so that $f^* = F$.
\end{corollary}

\begin{proof}

As a consequence of  results of  Masuda and Panov in \cite{MP}, any symplectic toric manifold whose integral cohomology ring is isomorphic to $\hh^*( (\CP^1)^n;\Z)$ is symplectomorphic to a symplectic Bott manifold; 
see  Proposition \ref{cube} in this paper for more details.
Therefore, we may assume that $(M,\omega)$ is a symplectic Bott manifold.

Since $\omega$ is rational, 
$F([\omega]) = \left[ \sum_i \lambda_i \pi_i^*(\omegaFS)\right]$ for some $\lambda_i \in \Q$ for all $i$.
Since $\omega$ is symplectic $[\omega]^n \neq 0$, and so $\lambda_i \neq 0$ for all $i$.
Since complex conjugation induces a symplectomorphism from $(\CP^1, \omegaFS)$
to $(\CP^1, -\omegaFS)$, we may assume without loss of generality that $\lambda_i > 0$ for all $i$; hence,  $(\CP^1)^n,\omega_\lambda)$ is also a symplectic Bott manifold.  

The claim now follows immediately from the theorem above.

\end{proof}

Another corollary can be obtained by building on a result of Illinskii \cite{I} (reproved in \cite{MP}):  Let
$X$ be a smooth toric variety of complex dimension $n$. If a circle subgroup of $(\C^\times)^n$ acts  semifreely with isolated fixed points, then 
$X$ is diffeomorphic to $(\CP^1)^n$.

\begin{corollary}
Let $(M,\omega)$  be a $2n$-dimensional symplectic toric manifold with rational symplectic form.
If a circle subgroup of $(S^1)^n$ acts semifreely with isolated fixed points, then $(M,\omega)$ is symplectomorphic to $((\CP^1)^{n}, \omega_\lambda)$ for some $\lambda \in \Q^n$ such that $\lambda_i > 0$ for all $i$.
\end{corollary}

Finally, as  a corollary of our proof, we obtain a complete classification (up to symplectomorphism) of symplectic  Bott manifolds with rational symplectic forms whose rational cohomology rings are  isomorphic to $\hh^*( (\CP^1)^n;\Q)$ up to symplectomorphism; see Corollary \ref{cor classification}.

We will now briefly explain how we use toric degenerations  to  construct the symplectomorphisms required for the proof of the main theorem. 
The main tool for constructing symplectomorphisms is a combination of results of Anderson and Harada-Kaveh.
Given a smooth projective variety $X$, of complex dimension $n$, equipped with a very ample line bundle, satisfying certain conditions, Anderson constructs a toric degeneration, 
i.e.,  a flat family $\pi \colon \mathfrak{X} \rightarrow \C$ with generic fiber $X$ so that the  special fiber $ X_0=\pi^{-1}(0)$ is a (not necessarily normal) toric variety (\cite[Theorem 2]{A}).
Harada and Kaveh observed that  one can build a symplectic form $\tilde{\omega}$ on 
$\mathfrak{X}$ and use the flow of a certain vector field to  obtain a surjective continuous map $\phi \colon (X, \omega) \rightarrow  (X_0,\omega_0:= \tilde{\omega}_{|X_0})$, which is a symplectomorphism when restricted to a particular dense open subset $U$ of $X$ (\cite[Theorem A]{HK}). 
In especially nice situations this map is actually a symplectomorphism from $X$ to $X_0$.

{\bf Organization} The paper is organized as follows. In Section \ref{sec degenerations} we explain how to obtain symplectomorphisms between symplectic toric manifolds using  toric degenerations. 
Section \ref{sec Bott manifolds} is devoted to symplectic Bott manifolds.  We construct the symplectomorphism between specific symplectic Bott manifolds in  Section \ref{sec tech}. Finally, the main result is  proved in Section \ref{sec proof}.

{\bf Acknowledgements.} The first author was supported by a fellowship SFRH/BPD/87791/2012 of the 
Funda\c{c}\~ao para a Ci\^encia e a Tecnologia (FCT, Portugal)
 during the first stage of the project and by the DFG (Die Deutsche Forschungsgemeinschaft) grant CRC/TRR 191 ``Symplectic Structures in Geometry, Algebra and Dynamics" during the second stage of the project.
The second author was partially supported by NSF grant DMS \#1206365,
and by a Simon's Foundation Collaboraton Grant.

\section{Symplectomorphisms coming from toric degenerations}\label{sec degenerations}
In this section we prove that two  symplectic toric manifolds
are symplectomorphic if their moment polytopes are related by a ``slide" (Proposition \ref{nice cond give sympl}).
Our  main tool is a  theorem of Harada and Kaveh which, in many cases,
gives a map, from a given smooth projective variety $X$ with integral symplectic form to toric variety $X_0$,  which is a symplectomorphism when restricted to an open dense subset.

Consider a
smooth projective variety $X$ of complex dimension $n$  equipped with a
very ample Hermitian line bundle $\calL$. 
Let $L:=\hh^0(X, \mathcal{L})$ denote the vector space of holomorphic sections. 
The Hermitian structure on $\calL$ induces a Hermitian structure on  $L$,
and hence a Fubini-Study form $\omegaFS$ on $\P(L^*)$.
Since the Kodaira embedding 
$\Phi_\calL \colon X \rightarrow \P (L^*)$ 
is holomorphic,
the pull-back $\omega :=\Phi_\calL^*(\omegaFS)$ is a K\"ahler form on $X$; moreover,
$[\omega]$ is the first Chern class of $\calL$.

Fix a coordinate system $(u_1,\dots,u_n)$ near a point $p \in X$,
that is, assume that $u_i \in \C(X)$ satisfies $u_i(p) = 0$ for all $i$,
and that $du_1,\dots,du_n$ are linearly independent.
Here,  $\C(X)$ denotes the ring of rational functions on $X$.
We define a function $\nu \colon \C(X)\setminus \{0\} \rightarrow \Z^n$ as follows: 
If a nonzero function $f \in \C(X)$ is regular at $p$, then nearby it can be
represented as a power series 
$\sum_{\alpha \in \Z^n_{\geq 0}} c_\alpha u_\alpha$.  Define
$$\nu(f) := \min \{\alpha \mid c_\alpha \neq 0\},$$
where the minimum is taken with respect to the lexicographical order.
More generally, if $f,g \in \C(X)$ are regular at $p$, define $\nu(f/g) = \nu(f) - \nu(g)$.
We call $\nu$ the {\bf valuation} associated to the  coordinate system.
As noted in \cite[Example 3.2]{HK}, this function
satisfies the axioms for a
{\em valuation with one-dimensional leaves}, as described in  \cite[Definition 3.1]{HK}.
In particular, $\nu(fg) = \nu(f) + \nu(g)$ for all $f,g \in \C(X) \setminus \{0\}$.
Moreover, given any finite dimensional subspace $E \subset \C(X)$, the cardinality of the image $ \nu(E \setminus \{0\})$
is the dimension of $E$ \cite[Proposition 3.4]{HK}.

Let $L^m$ denote
the image of $L^{\otimes m}$ in $\hh^0(X, \mathcal{L}^{\otimes m})$, and   let
$R=\C[X]=\oplus_{m\geq 0}L^m$ be 
the homogeneous coordinate ring of $X$ with respect to the embedding Kodaira embedding $\Phi_\calL$.
Choose a non-zero element $h \in L$ and identify  
$L^m$ with a subspace of $\C(X)$ by sending $f \in L^m$ to $f/h^m \in \C(X)$. 
Then we  can define an additive  semigroup
$$S=S(\nu,\calL)=\bigcup _{m\geq 0} \{ (m, \nu(f/h^m)\,|\,f \in L^m \setminus \{0\}\,\} = \bigcup_{m \geq 0} \{m\} \times S_m,$$
where $S_m:=\{  \nu(f/h^m) \mid f \in L^m \setminus \{0\}\,\}$. 
The dependence on $h$ is minor: replacing $h$ by $h'$ simply shifts $S_m$ by  $m\, \nu (h/h')$. Hence, we will usually omit $h$.
The {\bf Okounkov body} associated to $S$ is 
$$ \Delta=\Delta(S)=\overline{\conv( \bigcup_{m>0} \{x/m\,|\, (m,x) \in S\})}\subset \R^n.$$
If
$S$ is finitely generated, then the Okounkov body
$\Delta$ is a rational polytope. 

Finally, given a finitely generated semigroup $S \subset \Z \times \Z^n$,
there is a natural $(\C^\times)^n$ action on  the variety $X_0 := \Proj \C[S]$,
making it into a  (not necessarily normal) toric variety

We will need a corollary of the  following theorem, which is a slight variant of Theorem 3.25 in  \cite{HK}. 

\begin{theorem}\label{HK} 
Let $X$ be a smooth projective variety of complex dimension $n$, 
equipped with a very ample Hermitian line bundle $\calL$, and let $\omega$ be the 
associated K\"ahler form.
Let $S$ be the semigroup associated to a valuation derived from
a coordinate system on $X$. 
Assume that $S$  is finitely generated;
let $\Delta$  and $X_0 := \Proj \C[S]$
be the associated Okounkov body and toric variety, respectively.
\begin{enumerate}
\item
There's a linear $(\C^\times)^n$ action on $\C^{N+1}$,
 an $(S^1)^n$ invariant symplectic form $\Omega$ on $\CP^N$ with moment map $\mu_{\Omega} \colon \CP^N \to \R^n$,
and an equivariant embedding 
$X_0   \hookrightarrow \CP^N$ such that $\mu_{\Omega} (X_0) = \Delta$.
\item There's a surjective continuous map  $\phi \colon X \to X_0$
that restricts to a symplectomorphism from a dense open subset $U \subset X$ to $(X_0)_\smooth$,
the set of smooth points in $X_0$, with the symplectic form induced from $\Omega$. 
\end{enumerate}
\end{theorem}

The proof of Theorem 3.25 given in \cite{HK} actually proves the above statement.
Unfortunately, this fact may not be immediately obvious to the casual reader,
and so we will sketch the proof.

\begin{proof} 
Anderson constructs a toric degeneration (in the sense of \cite[Definition 2]{HK}), i.e. a flat family $\pi \colon \mathfrak{X} \rightarrow \C$ with generic fiber $X$ and a special fiber $ X_0=\pi^{-1}(0)$, a (not necessarily normal) toric variety; see \cite[Corollary 5.2]{A} or \cite[Corollary 3.14]{HK}.

In \cite[Section 3.3]{HK},
Harada and Kaveh fix a positive multiple $\Omega$ of the Fubini-Study K\"ahler form
on  $\CP^N $, construct an embedding 
$\mathfrak X \hookrightarrow \CP^N \times \C$ inducing a symplectic structure $\tilde \omega$ on the smooth part of $\mathfrak{X}$ as a pull back of $\Omega \times \frac{i}{2} \operatorname{d}z \wedge \operatorname{d} \bar{z}$, and
show that these satisfy the assumptions (a)-(d) of \cite[Theorem 2.19]{HK}.
See \cite[Corollary 3.15, Remarks 3.17 and 3.19, and Proposition 3.24]{HK} for details.
In particular, assumption (b) implies that
there exists a linear action of $(\C^\times)^n$ on $\C^{N+1}$ so
that the induced embedding of $X_0 \hookrightarrow \CP^N \times \{0\}$
is equivariant.
The weights that occur with multiplicity at least one in this representation
are $S_d := \{ u \in \Z^n \mid (d,u) \in S \}$ for some natural number $d$; moreover,
$\conv (S_d )= d \Delta$;
(see \cite[Section 3.3.5]{HK}).
Moreover, $\Omega$ is chosen so that  $\mu_{\Omega} (\CP^N) = \Delta$,
where $\mu_{\Omega} \colon \CP^N \to \R^n$ is the moment map for $\Omega$.
Furthermore, \cite[Remark 3.17]{HK} implies that $X_0$ embeds in $\CP^N$
as the closure of the $(\C^*)^n$ orbit of $(1,\dots,1)$,
and so  $\mu_{\Omega} (X_0) = \Delta$.
This proves (1).

Furthermore, by part (1) of \cite[Theorem 2.19]{HK}, there exists a surjective continuous map $\phi \colon X \to X_0$ which is a symplectomorphism when restricted to a dense open subset $U \subset X$. 
Reading the proof of \cite[Theorem 2.19, part (1)]{HK}, 
we see that $\phi(U) = U_0$, the set defined just before \cite[Lemma 2.9]{HK}. Finally,
by \cite[Corollary 2.10]{HK}, $U_0$ is the smooth locus of $X_0$.  
This proves (2).

\end{proof}

\begin{remark}\label{topol after HK}
In the setting of Theorem~\ref{HK}, the facts that $X$ is Hausdorff,  $\phi$ is continuous,
and $\phi$ restricts to a homeomorphism from $U$ to $(X_0)_{\smooth}$,
together imply that $U =\phi^{-1}((X_0)_{\smooth}) \cap \overline{U}$.
Since $U$ is dense in $X$, this implies that $U=\phi^{-1}((X_0)_{\smooth})$.
\end{remark}

We will only need the following corollary.

\begin{corollary}\label{existence sympl}
Let $X$ be a smooth projective variety of complex dimension $n$,
equipped with a very ample Hermitian line bundle $\calL$,
and let $\omega$ be the associated K\"ahler form.
Let $S$ be the semigroup associated to a valuation
derived from a coordinate system on $X$.
Finally, let $(X_\Delta,\omega_\Delta)$ be a symplectic
toric manifold with integral  moment polytope  $\Delta$.
If 
$$S=\Cone (\{1\} \times \Delta)  \cap (\Z \times \Z^n),$$
then $(X,\omega)$ is symplectomorphic to $(X_\Delta,\omega_\Delta)$.
\end{corollary}

\begin{proof}
Because $S= \Cone(\{1\} \times \Delta) \cap (\Z\times \Z^n)$ is the set of integral points in a rational polyhedral cone,
 $S$ is finitely generated by Gordan's Lemma. 
Hence,  we can apply Theorem \ref{HK}. 
Moreover, since $\Delta$ is a smooth polytope, the toric variety $X_0 := \Proj \C[S]$ is smooth.
Therefore, since $\Delta$ is the Okounkov body associated to $S$, 
Theorem \ref{HK} and  Remark \ref{topol after HK} together
imply that  $(X,\omega)$ is symplectomorphic to $(X_0, \Omega|_{X_0})$, which is a
 symplectic toric manifold with moment polytope $\Delta$. The claim then follows from Delzant's theorem.
\end{proof}

\subsection{Examples of toric degenerations}\label{examples}

In this subsection, we give  a concrete description of the
semigroup associated to specific coordinate system  on smooth toric
varieties.
We then combine this description and the results above to give a criterion which guarantees that
certain symplectic toric manifolds are symplectomorphic -- even  though, in general, 
there is no affine transformation between their moment polytopes.
As an immediate consequence, we recover the known symplectomorphisms between
different Hirzebruch surfaces.

Fix a smooth full dimensional polytope  $P \subset \R^n$; 
let $X_P$ be the smooth toric variety associated to $P$,
and let $\omega_P$ be an invariant symplectic form with moment map $\mu_P$ 
so that $\mu_P(X_P) = P$.
Assume also  that  $P$ is {\bf integral},
that is, the vertices lie in $\Z^n$. 
Then there exists a  holomorphic line bundle $\mathcal{L}_P$ over 
 $X_P$ such that $c_1(\mathcal{L}_P) = [\omega_P$].
Moreover, $\calL_P$ is very ample;
see, for example, \cite[Propositions 2.4.4 and 6.1.4]{CLS}.
Finally,  the $(\C^\times)^n$  action on $X_P$ 
lifts to an action on $\calL_P$; the induced representation on $L_P := \hh^0(X_P, \calL_P)$
has one dimensional weight spaces with  weights $P \cap \Z^n$ (\cite{D}, see also \cite{H}).

Assume additionally that
$P$ is {\bf aligned} with the positive orthant $\R^n_{\geq 0}$, that is, 
$P$ is equal to $\R^n_{\geq 0}$ on some open neighborhood of $0$. (In particular, this implies that $P$ is full dimensional.)
Let $h \in L_P$ be a non-zero section in the zero weight space,
and let $g_i \in L_P$ be a non-zero section in the weight space corresponding
to $e_i$ for all $1\leq i \leq n$, where  $e_1,\dots,e_n$ is the standard  basis for $\R^n$.
We can then define $f_i \in \C(X_P)$ by $f_i = g_i/h$ for all $i$.
Using, for example,  the description of $L_P$ given in \cite{H}, it is easy to check  that
 $f_1,\dots,f_n$ defines a coordinate system on an open dense subset of $X_P$ containing
the point $\mu_P\inv(0)$.  Moreover, the monomials $\{f_1^{p_1} \dots f_n^{p_n}\}_{p \in P}$ form a basis for 
the subspace $\{ g/h \mid g \in L_P \} \subset \C(X_P)$. 

Finally, assume that the polytope  $P$ is {\bf normal}, that is,
for every  positive integer $m$ and every $p \in mP \cap \Z^n$ there exist $p_1, \ldots,p_m \in P \cap \Z^n$ such that 
$p=p_1+\dots+p_m$. 
Given an integer $m > 1$,  
let $L^m_P$ be the image of $L_P^{\otimes m}$ in $\hh^0(X_P, \calL_P^{\otimes m}).$ 
Then the monomial $f_1^{p_1} \dots f_n^{p_n}$ lies in the subspace $\{ {g}/{h^m} \mid g \in L_P^m \} \subset \C(X_P)$
exactly if $p$ can be expressed as a sum of $m$ integral points of $P$; moreover, these monomials
give a basis for that subspace.
In general (i.e. if the polytope $P$ is not necessarily normal), it is hard to describe this basis explicitly.
However, since  $P$ is normal, the monomials $\{f_1^{p_1} \dots f_n^{p_n}\}_{p \in mP}$ form a basis for 
the subspace $\{ g/h^m \mid g \in L^m_P \} \subset \C(X_P)$.

\begin{example}\labell{trivial}
Let $\nu$ be the valuation derived from the
coordinate system $f_1,\dots,f_p$.
In this case, $\{ \nu(g/h^m) \mid    g \in L_P^m \setminus \{0\}\}  =mP \cap \Z^n$, that is,
the semigroup $S$ associated to $\nu$ is  $\Cone(\{1\} \times P) \cap (\Z \times \Z^n)$.
\end{example}

As we see from the example above,  if we use the valuation derived from the coordinate system $f_1,\dots,f_n$, then
Proposition~\ref{existence sympl} only yields the trivial statement that $(X_P,\omega_P)$
is symplectomorphic to itself.
Therefore, we will modify this system of coordinates.
Fix an non-negative integer $c$ and integers $1 \leq k < \ell \leq n$. Consider  
the coordinate system $u_1,\dots,u_n$ near $p$, where
\begin{equation}\labell{system}
u_i = 
\begin{cases}
f_i &  i \neq k \\
f_k - f_\ell^c & i = k.
\end{cases}
\end{equation}

As we show below,
the semigroup  associated to the valuation derived from this
coordinate system  can be obtained by sliding the integral points of  $\Cone( \{1\} \times P)$
as far possible in
the direction $-e_k + c e_\ell$, while staying in $\Z \times \Z_{\geq 0}^n$.

\begin{definition}\labell{defineF}
Given a subset $Q$ of $\Z_{\geq 0}^n$ and a vector  $w \in \Z^n \setminus \Z^n_{\geq 0}$, 
let $\mathcal{F}_w(Q)$ 
be the set  formed by sliding the  
points of $Q$ as far as possible in the direction $w$, while staying in $\Z^n_{\geq 0}$.
More formally, 
let  $\mathcal{F}_w(Q)$ 
be the unique subset of $\Z^n_{\geq 0}$ such that
\begin{enumerate}
\item $Q \cap \af$ and $\mathcal{F}_w(Q) \cap \af$  have the same  cardinality
for every affine line $\af$ that is parallel to $w$.
\item If $p \in \mathcal{F}_w Q$ and $p + w \in \Z_{\geq 0}^n$, then 
$p + w \in \mathcal{F}_w(Q)$.
\end{enumerate}
\end{definition}

\begin{lemma}\label{sliding}
Consider the toric variety $X_P$ and  holomorphic line bundle $\calL_P$
associated to a smooth  normal integral  polytope $P \subset \R^n$ that's aligned with $\R^n_{\geq 0}$.
Fix a non-negative integer $c$ and integers $1 \leq k < \ell \leq n$;
let $\nu$ be the valuation derived from the local coordinate system
defined in \eqref{system}, and  let $S$ be the associated semigroup.
For each $m > 0$,
$$S_m = \F_{-e_k+c e_\ell}(m P \cap  \Z^n).$$
\end{lemma}

\begin{proof}

Fix an affine line $\af$ parallel to $-e_k + c e_\ell$ such that $P\cap \Z^n \cap \af\neq \emptyset$.
Since $P$ is convex, $mP\cap \af$ is  an interval.
Hence,
\begin{equation*}
mP\cap \Z^n\cap \af=\{p +t(e_k - c e_\ell)  \mid t \in [0,N]\cap \Z\},
\end{equation*}
for some $N \in \Z_{\geq 0}$ and $p=(p_1,\ldots, p_n) \in P\cap \Z^n$.
Fix  $t \in [0,N]\cap \Z$. 
On the one hand, the function  $$f:=(f_k-f_\ell^c)^{t}f_\ell^{p_\ell- c t} \prod_{i \neq \ell} f_i^{p_i}
= \sum_{j=0}^{t} (-1)^{t-j}\,\binom{t}{j}\,f_k^{p_k + j} \,f_\ell^{p_\ell-cj} \prod_{i \neq k,\ell} f_i^{p_i}$$
is of the form $g/h$ for some $g \in L^m_P$ because
$P$ is normal and the exponent of each monomial  in the final sum lies in $mP \cap \Z^n \cap \af$.
On the other hand, 
\begin{equation*}
f=u_k^{t}
u_\ell^{p_\ell -  ct} (u_k + u_\ell^c)^{p_k}
\prod_{i \neq k, \ell} u_i^{p_i}
= \sum_{j=0}^{p_k} \, \binom{p_k}{j} u_k^{t+j} u_\ell^{p_\ell - ct + c p_k  - c j}
\prod_{i \neq k, \ell} u_i^{p_i},
\end{equation*}
and so $\nu(f) = p + (t - p_k)(e_k - ce_\ell)$.
Since this is true for all $t \in [0,N] \cap \Z$,
\begin{equation*}
\{p +(t - p_k) (e_k - c e_\ell) \mid t \in [0,N]\cap \Z\} \subseteq \{ \nu(g/h^m) \mid g \in L^m_P \setminus \{0\} \}.
\end{equation*}
Finally, as explained in the beginning of this section, 
the cardinality of 
the set $\{ \nu(g/h^m) \mid g \in L^m_P \setminus \{0\} \}$
is equal to the dimension of $L^m_P$, which
in turn is equal to the cardinality of the intersection  $mP \cap \Z^n$.
Hence, there are no additional points in the set $S_m =
\{ \nu(g/h^m) \mid g \in L^m_P \setminus \{0\} \}$.
The claim then follows immediately from  Definition \ref{defineF}.
\end{proof}

\begin{proposition}\label{nice cond give sympl}
Let $P$ and $\Delta$ be integral smooth polytopes
in $\R^n$ that are aligned with $\R^n_{\geq 0}$.
Let $(X_P, \omega_P)$ and $(X_\Delta,\omega_\Delta)$
be  the  symplectic toric manifolds associated  to $P$ and
$\Delta$, respectively.
Assume that there exists a non-negative integer $c$ and integers $1 \leq k < \ell \leq n$ such that
\begin{equation}\labell{key}
\mathcal{F}_{-e_k + c e_\ell} (mP \cap \Z^n) = m\Delta \cap \Z^n
\end{equation}
for all $m \in \Z_{>0}$.
Then $(X_P,\omega_P)$ is symplectomorphic to $(X_\Delta,\omega_\Delta)$.
\end{proposition} 

\begin{proof}
First observe that without loss of generality we can assume that $P$ is a normal polytope.
To see this, recall that, since $n \geq n-1$, the $n$-th dialate of  any smooth $n$-dimensional polytope  is normal; see, for example,
 \cite[Theorem 2.2.12]{CLS}.
Obviously the assumption \eqref{key} still holds if we replace $P$ by  $nP$ and $\Delta$ by $n\Delta$.
Moreover there is an equivariant biholomorphism from $X_P$ to $X_{nP}$ that
identifies  the equivariant holomorphic line bundle
$\mathcal{L}_{nP} \to X_{nP}$ 
with the tensor product $\mathcal{L}_P^{\otimes n} \to X_P$, and
identifies the symplectic form $\omega_{nP} \in\Omega^2(X_{nP})$
with the form $n \omega_P \in \Omega^2(X_P)$.
Therefore, if $(X_{nP},\omega_{nP})$ is symplectomorphic to $(X_{n\Delta},\omega_{n\Delta})$ then also $(X_P,\omega_P)$ is symplectomorphic to $(X_\Delta,\omega_\Delta)$.

Thus we assume that $P$ is normal. 
Use $c, \ell$ and $k$ to
define a coordinate system as in \eqref{system},
and hence  define a valuation $\nu$ and semigroup $S$.
By  Lemma~\ref{sliding} and \eqref{key} we have that
$$S=\Cone(\{1\} \times \Delta) \cap (\Z \times \Z^n).$$
The claim now follows immediately
from Proposition~\ref{existence sympl}.
\end{proof}

Unfortunately, even just the condition  that $\mathcal{F}_{-e_k + ce_\ell}(P \cap \Z^n)$  must
be the set of integral points in a smooth polytope is quite strong; 
it may not be the set of integral points of any convex polytope.
For example, consider the polytopes
\begin{gather*}
P=\{p \in \R^2 \mid
  0\leq \langle p,\ e_1 \rangle\leq 2,\ 0\leq \langle p, e_2 \rangle \leq 2\}, \\
P' =\{p \in \R^2 \mid
  0\leq \langle p,\ e_1 \rangle\leq 1,\ 0\leq \langle p, e_2 \rangle,\  \langle p, 4 e_1 + e_2 \rangle \leq 6\}.
\end{gather*}
In this case, $\conv(\F_{(-1,2)}(P \cap  \Z^2)) = P'$;
however, $\F_{(-1,2)}(P \cap \Z^2) \neq  P' \cap \Z^2$ because the former does not contain  the point $(1,1)$.
(This doesn't violate Lemma~\ref{trapezoid} below because it doesn't satisfy the correct inequalities.)
Luckily, there are some interesting  examples where the conditions of the above  proposition
are satisfied.

\begin{lemma}\labell{trapezoid}
Given integers $A_2^1,$  $\lambda_1,$ and $\lambda_2$,  such that
$\lambda_1 > 0$, $\lambda_2 > 0$, and $\lambda_2 > A_2^1 \lambda_1$,
consider the trapezoid
\begin{equation} \labell{P1}
P=\{p \in \R^2 \mid  0\leq \langle p, e_1 \rangle\leq \lambda_1,\ 0\leq \langle p, e_2 \rangle,\ \langle p,e_2 + A^1_2 e_1\rangle \leq \lambda_2\}.
\end{equation}
Given an integer $c$ such that  $ c \geq A_2^1$  and  $\lambda_2 > c \lambda_1$, 
$$ \F_{(-1,c)} (P \cap \Z^2)=\tilde P \cap \Z^2,$$
where
$\tA_2^1 = 2c - A_2^1,$ $\tlambda_1 = \lambda_1$, and $\tlambda_2 = 
\lambda_2 + (c - A_2^1) \lambda_1$, and
\begin{equation}\labell{P2}
\Tilde P=\{p \in \R^2 \mid  0\leq \langle p, e_1 \rangle\leq \tlambda_1,\ 0\leq \langle p, e_2 \rangle,\ \langle p,e_2 + \tA^1_2 e_1\rangle \leq \tlambda_2\}.
\end{equation}
\end{lemma}

\begin{proof}
The polytope $\Tilde P \cap \Z^2$ satisfies condition (2) of Definition~\ref{defineF}
because $\tA_2^1 \geq c$.
Moreover, if  $\tA_2^1 = A_2^1$ then $P = \Tilde P$, and so condition (1) is obvious.
Thus, we may 
assume that $\tA_2^1 > c > A_2^1$ and
focus on proving condition (1).

First, consider the half-plane  $\hh^\leq := \{ p \in \R^2 \mid \langle p, e_2 + c e_1 \rangle  \leq   \lambda_2 \}$. 
Since $c > A_2^1$, every $p \in H^\leq$ with $\langle p, e_1 \rangle \geq 0$
satisfies 
$\langle p, e_2 + A_2^1 e_1 \rangle \leq \lambda_2$; moreover,
every $p \in \hh^\leq$ with $\langle p, e_1 \rangle \leq \lambda_1$
satisfies
$\langle p, e_2 + \tA_2^1 e_1 \rangle \leq \tlambda_2$.
Therefore, the intersections  $P \cap \hh^\leq$ and $\Tilde P \cap \hh^\leq$
are both equal to the  trapezoid $Q$ with vertices
$(0,0)$, $(\lambda_1,0)$, $(\lambda_1, \lambda_2 - c \lambda_1)$, and $(0,\lambda_2)$.
Hence, condition (1)  holds for lines in this half plane.

Next,
consider the half-plane $\hh^\geq := \{ p \in \R^2 \mid \langle p, e_2 + c e_1 \rangle \geq \lambda_2 \}$.
Since $c > A_2^1$, every $p \in \hh^\geq$ with 
$\langle p, e_1 + A_2^1  e_2 \rangle \leq \lambda_2$ satisfies $\langle p, e_1 \rangle \geq 0$;
moreover,  every $p \in \hh^\geq$ with $\langle p, e_1 + \tA_2^1 e_2 \rangle \leq \tlambda_2$ satsfies $\langle p, e_1 \rangle \leq \lambda_1$.
Additionally, since $\lambda_2 > 0$ and $\lambda_2 > c \lambda_1$,
every $p \in  \hh^\geq$
with  $0 \leq \langle p, e_1 \rangle \leq \lambda_1$ satisfies $\langle p, e_2 \rangle \geq 0$.
Since $c >  A_2^1$ and $\lambda_1 > 0$, this implies that the
intersection $P \cap \hh^\geq$  is the  triangle $T$ with vertices $(0,\lambda_2)$,
$(\lambda_1, \lambda_2-c \lambda_1)$,
and $(\lambda_1, \lambda_2 - A_2^1 \lambda_1)$,
and the intersection  $\Tilde P \cap \hh^\geq$ is the triangle $\tilde T$ with vertices $(0,\lambda_2)$,
$(\lambda_1, \lambda_2-c \lambda_1)$,
and $(0, \lambda_2 + c \lambda_1 - A_2^1 \lambda_1)$.
Define an affine transformation $B \colon \R^2 \to \R^2$  by
$$B(p_1,p_2) = ( \lambda_1 - p_1,  2c p_1 + p_2  - c \lambda_1).$$
By inspection, $B(T) = \Tilde T$.
Additionally, since $B$ is the composition of a  unimodular integral linear transformation
with  translation by an integral vector, 
$B(\Z^2) = \Z^2$.
Finally, $B(\af) = \af$ for every affine line $\af$ parallel to  $- e_1 + c e_2$.
Hence, $B$ induces a bijection from 
$P \cap \Z^2 \cap \af$ to $\Tilde P \cap \Z^2 \cap \af$ for every such $\af$ in $\hh^\geq$,
and so condition (1) also holds for lines in this half plane.
\end{proof}

\begin{example}\labell{Hirzebruch}
Given an integer $m$, let  $\calO(m) \to \CP^1$ be the holomorphic line bundle
$\calO(m) := \big(\C^2 \setminus \{0\} \times \C\big)/\C^\times$, where $\C^\times$ acts
by $\alpha \cdot (x_1,x_2,z) = (\alpha  x_1, \alpha x_2, \alpha^m z)$.
Given an  non-negative integer $m$, consider the Hirzebruch surface
$$\Sigma_{m} := \P(\calO(0) \oplus \calO(-m)) \simeq \big(\C^2 \setminus \{0\}\big)^2/(\C^\times)^2,$$
where  $(\alpha_1,\alpha_2) \cdot (x_1,x_2, y_1,y_2) =
(\alpha_1 x_1, \alpha_1 x_2,   \alpha_2 y_1, \alpha_1^{-m} \alpha_2 y_2)$.
Hirzebruch surfaces are the simplest non-trivial Bott manifolds; 
the torus  $(\C^\times)^2$ acts on $\Sigma_{m}$  by
$(\alpha_1, \alpha_2) = [\alpha_2 x_1, x_2, \alpha_1 y_1, y_2]$.
Given integers $A_2^1 \geq 0$, $\tA_2^1 \geq 0$,
$\lambda_1 >0$, $\tlambda_1$,  $\lambda_2 > A_2^1 \lambda_1 \geq 0$, and $\tlambda_2 > \tA_2^1 \tlambda_1 \geq 0$,
there exist $(S^1)^2 \subset (\C^\times)^2$ invariant integral symplectic forms $\omega  \in \Omega^2(\Sigma_{A_2^1})$
and $\tomega \in \Omega^1(\Sigma_{\tA_2^1})$, and moment maps $\mu \colon \Sigma_{A_2^1} \to \R^n$ and 
$\Tilde \mu \colon \Sigma_{A_2^1} \to \R^n$, so  that $\mu(\Sigma_{A_2^1}) = P$ 
and  ${\Tilde \mu}(\Sigma_{A_2^1}) = \Tilde P$, where $P$ and $\Tilde P$ are the polytopes described in
\eqref{P1} and \eqref{P2}.
Then  $(\Sigma_{A_2^1},\omega)$ and $(\Sigma_{\tA_2^1}, \tomega)$ are symplectomorphic exactly
if $A_2^1$ and $\tA_2^1$ have the same parity, $\tlambda_1 = \lambda_1$, and $\lambda_2 - \frac{1}{2} A_2^1 \lambda_1 = \tlambda_2 - \frac{1}{2} \tA_2^1 \tlambda$.
When these equations are satisfied, Proposition~\ref{nice cond give sympl} and
Lemma~\ref{trapezoid} together  immediately give another proof that these manifolds
are symplectomorphic.
\end{example}

We will use the following easy consequence of the previous lemma. 
\begin{lemma} \label{cut trapezoid}
Given $A_2^1, \lambda_1, \lambda_2  \in \Z$ and 
$\lambda', \lambda'' \in \Q \cup \{ \pm \infty\}$,
let $Q$ be the polytope consisting of $p \in \R^2$ such that
\begin{multline*}
 0\leq \langle p, e_1 \rangle\leq \lambda_1,\ \  0\leq \langle p, e_2 \rangle,\ \ \langle p,e_2 +  A^1_2 e_1\rangle \leq \lambda_2, \\
\lambda'\leq \langle p, e_2 + c e_1  \rangle \leq \lambda''.
\end{multline*}
Given a non-negative integer $c$  such that $\lambda_2 > c \lambda_1,$
let  $\Tilde Q$ be the polytope consisting of $p \in  \R^2$ such that
\begin{multline*}
 0\leq \langle p, e_1 \rangle\leq \tlambda_1,\ 0\leq \langle p,e_2\rangle,\  \langle p,e_2 + \tA^1_2 e_1\rangle \leq \tlambda_2, \\
\lambda'\leq  \langle p, e_2 +  c e_1  \rangle \leq \lambda'', 
\end{multline*}
where $\tA_2^1 = 2c - A_2^1,$ $\tlambda_1 = \lambda_1$, and $\tlambda_2 = 
\lambda_2 + (c - A_2^1) \lambda_1$.
If $\tA_2^1 \geq A_2^1$, then $$\F_{(-1,c)} (Q \cap \Z^2)=\Tilde Q \cap \Z^2.$$
\end{lemma}

\begin{proof}
Assume that $\tA_2^1 \geq A_2^1$, and so $c  \geq A_2^1$.
If $\lambda_1 < 0$, the claim is obvious because $Q$ and $\Tilde Q$ are both empty. 
Similarly, if $\lambda_1 = 0$, then $Q = \Tilde Q \subset \R \times \{0\}$.
Hence, we may assume that $\lambda_1 > 0$.
Therefore,  since $c \geq 0$ and $c \geq A_2^1$, 
the assumption that $\lambda_2 > c \lambda_1$ implies
that $\lambda_2 > A_2^1 \lambda_1$ and $\lambda_2 > 0$ as well.
Therefore, 
by Lemma~\ref{trapezoid}
 $$\F_{(-1,c)} (P \cap \Z^2)=\tilde P \cap \Z^2,$$
where  $P$ and $\Tilde P$ are the polytopes defined in that lemma.
Moreover,
\begin{gather*}
Q = P \cap  \{p \in \R^2 \mid \lambda'\leq \langle p, e_2 + c e_1  \rangle \leq \lambda''\}, 
\\
\Tilde Q =  P' \cap  \{p \in \R^2 \mid \lambda'\leq \langle p, e_2 + c e_1  \rangle \leq \lambda''\}. 
\end{gather*}
Therefore, the claim follows immediately from the fact that $\langle -e _1 + c e_2, e_2 + c e_1 \rangle = 0.$
\end{proof}

\section{Symplectic Bott manifolds}\label{sec Bott manifolds}\labell{s:bm}

In this section, we formally define symplectic Bott manifolds and study some elementary properties of these manifolds.
In particular, we describe the cohomology ring of Bott manifolds and
determine which cohomology classes admit invariant K\"ahler forms.
We also use Delzant's theorem to construct symplectomorphisms between certain symplectic Bott manifolds.

To simplify notation,  let $\n$ denote the set $\{1,\ldots, n\}$  for any $n \in \Z_{>0}$,
and let $\mat$ denote the set of $n \times n$ integral matrices.
Let $A \in \mat$ be a strictly upper-triangular matrix,
that is,  let $A^i_j$ be an integer for all $i,j \in \n$
such that  $A^i_j = 0$  if $i \geq j$.
Fix $\lambda \in \R^n$,  
and define a polytope $\Delta=\Delta(A, \lambda)$ as follows:
$$ \textstyle \Delta = \big\{ p \in \R^n \ \big| \ \langle p, e_j\rangle \geq 0 \mbox{ and }
\big\langle p, e_j + \sum_i A^i_j e_i \big\rangle \leq \lambda_j \ \forall \ 1 \leq j  \leq n\big\}.$$
Given subsets $J, K \subseteq \n$, define (possibly empty) faces
$$
\hat{F}_J =  \bigcap_{j \in J} \{\langle p, e_j \rangle = 0 \} \cap \Delta
\mbox{ and }
 F_K = \bigcap_{k \in K} \big\{  
\big\langle p, e_k + \sum_i A^i_k e_i \big\rangle =  \lambda_k  \big\} \cap \Delta.
$$
If $A$ is the zero matrix and $\lambda = (1,\dots,1)$, then
$\Delta$ is the hypercube $[0,1]^n$, 
and the face $F_J \cap \hat F_K$ of $[0,1]^n$ is non-empty exactly if
$J \cap K = \emptyset$.
More generally,  we say that $\Delta = \Delta(A,\lambda)$ is a {\bf Bott polytope} if
\begin{equation*}
F_J \cap \hat{F}_K \neq \emptyset \iff
J \cap K = \emptyset \quad
 \forall \ J,K \subseteq \n.
\end{equation*}
In this case,  each vertex of $\Delta$ lies  on exactly $n$ defining hyperplanes;
moreover, the primitive integral vectors perpendicular to these hyperplanes form a
$\Z$ basis for $\Z^n$.
Therefore, $\Delta$ is an $n$-dimensional smooth polytope, and $F_i$ and $\Hat F_i$
are facets for each $i$.
(In particular, $\Delta$ is combinatorially equivalent to the hypercube $[0,1]^n$.)

Fix a strictly upper triangular matrix $A \in \mat$.
For each $j \in \n$,
let $\Hat v_j = e_j$ and $ v_j = -e_j - \sum_i A_j^i e_i$.
Define 
$\pi \colon \R^{2n} \to \R^n$  by
$$\pi(\Hat \alpha_1,\alpha_1 ,\dots,\Hat \alpha_n, \alpha_n) = \sum_j \Hat \alpha_j \Hat v_j + \alpha_j  v_j.$$
Let $\iota \colon \fk \to \R^{2n}$ be the natural inclusion the kernel of $\pi$, 
let $K \subset (S^1)^{2n}$ be the subtorus with Lie algebra $\fk$, and let $K_\C \subset ((\C^\times)^2) \simeq (\C^\times)^{2n}$ be the complexification of $K$.
Then the quotient
$$ M_A := (\C^2 \setminus \{0\})^n/K_\C$$  is a Bott manifold;
the $(\C^\times)^n$ action on $M_A$ is  induced by the
inclusion $(\C^\times)^n \simeq (\C^\times \times \{1\})^n \hookrightarrow (\C^\times)^{2n}$.
We call $M_A$ the {\bf Bott manifold} associated to $A$.
Conversely, if $M$ is any Bott manifold then -- for the appropriate identification of the
torus acting on $M$ with $(\C^\times)^N$ -- there exists a strictly upper triangular matrix
$A \in \mat$ such that  $M$ is equivariantly biholomorphic to $M_A$.

Now fix $\lambda \in \R^n$ such that $\Delta = \Delta(A,\lambda)$ is a Bott polytope.
In this case,  $\Hat v_j$ and $v_j$ are  the
inward primitive integral vectors perpendicular to the facets $\Hat F_j$ and $F_j$, respectively.
The standard $(S^1)^{2n}$ action on $\C^{2n}$ is Hamiltonian with moment map $$\textstyle(\Hat z_1,z_1,\dots,\Hat z_n, z_n) \mapsto
\left(\frac{1}{2} |\Hat z_1|^2, \frac{1}{2} |z_1|^2, \dots, \frac{1}{2} |\Hat z_n|^2, \frac{1}{2} |z_n|^2\right).$$
Let $(M_\Delta,\omega_\Delta)$ be the  symplectic reduction of $\C^{2n}$ by $K\subset (S^1)^{2n}$
at $\iota^*(0,\lambda_1, \dots,  0, \lambda_n) \in \fk^*$, where $\iota^* \colon \R^{2n} \to \fk^*$ is the dual map.
Then the  inclusion $(S^1)^n \simeq (S^1 \times \{1\}) \hookrightarrow  (S^1)^{2n}$ induces
a Hamiltonian $(S^1)^n$ action  action on $(M_\Delta,\omega_\Delta)$  with moment  polytope $\Delta$.
Moreover, since $\Delta$ is a Bott polytope,
the manifold $M_\Delta$ is equivariantly diffeomorphic to the Bott manifold $M_A$ defined above,
Moreover, $\omega_\Delta$ is a K\"ahler form with respect to this complex structure.
(See, e.g., \cite{H}.)
Conversely, if $\omega$ is any $(S^1)^{n} \subset (\C^\times)^n$ invariant K\"ahler form on $M_A$, then the $(S^1)^{n}$ action on $M_A$ is Hamiltonian with moment image
$\Delta(A,\lambda)$ for some $\lambda \in \R^n$.
This justifies the following definition.

\begin{definition}
A {\bf symplectic Bott manifold} is the symplectic toric manifold $(M,\omega,\mu)$ associated
to a Bott polytope $\Delta(A,\lambda)$  via the above construction. 
In this case we say that $(A, \lambda)$ {\bf defines}  $(M,\omega, \mu)$.
\end{definition}

We need the following  special case of a classical result of Danilov \cite{D}.
\begin{lemma}\label{cohomology}
Let $(M,\omega,\mu)$ be the symplectic Bott manifold associated to a strictly upper triangular
matrix $A \in \mat$ and $\lambda \in \R^n$. Then
\begin{equation*}
\textstyle
\hh^*(M;\Z) \simeq \Z[x_1,\dots,x_n]/\big(x_i^2 + \sum_j A^i_j x_i x_j \big).
\end{equation*}
Moreover, $x_i \in \hh^2(M;\Z)$ is  Poincar\'e  dual to the moment preimage $\mu\inv(F_{i})$, $x_k + \sum_i A_i^i x_i$ is Poincar\'e dual to  $\mu\inv({\Hat F})$, and $[\omega]=\sum_i \lambda_i x_i$. 
\end{lemma}

Let $M_A$ be the Bott manifold associated to a strictly upper triangular matrix $A \in \mat$, and fix $\lambda \in \R^n$.
Our next proposition give a criteria for whether $(A,\lambda)$ defined a symplectic Bott manifold,
or equivalently, whether there is an $(S^1)^n$ invariant K\"ahler form on $M_A$ in the cohomology class $\sum_i \lambda_i x_i$.
Note, in particular, that  we can always find some  $\lambda \in \R^n$ 
that satisfies these conditions:
Choose $\lambda_{i+1} >> \lambda_i$ for all $i$.

Given a non-empty subset $I \subseteq \n$,
let  $\max(I)$ denote the maximal element of $I$.

\begin{proposition}\label{kahler condition}
Let $A \in \mat$ be  strictly upper-triangular and fix $\lambda \in \R^n$. 
Then $(A,\lambda)$ defines a symplectic Bott manifold exactly if
$\Xi(I) > 0$ for every non-empty subset $I \subseteq \n$,
where\footnote{
In particular, the summand associated to $\{\max(I)\}$ is $\lambda_{\max(I)}$.}
$$\Xi(I) := 
 \sum_{i_0 < \dots < i_m }(-1)^{m}  \lambda_{i_0} A^{i_0}_{i_1}  \dots A^{i_{m-1}}_{i_m},$$
where the sum is taken over subsets  $\{i_0,\dots,i_m\} \subseteq I$
such that $i_0 < \dots < i_m = \max(I)$. 
\end{proposition}

\begin{proof}
Let $\Delta$ be the polytope associated to $(A,\lambda)$.
Given a (possibly empty) subset $I \subseteq \n$, 
let $p_I$ be the unique point  
in the intersection of hyperplanes
$$ \textstyle \{ \eta \in \R^n \mid   \langle \eta, e_j \rangle = 0  \ \forall j \in I^c  \mbox{ and }
    \langle \eta, e_j + \sum_i A^i_j e_i  \rangle = \lambda_j \ \forall j \in I \}.$$
By induction on $j$, 
\begin{equation}\label{Xi1}
\langle p_I,e_j \rangle = \Xi([j] \cap I) \quad \forall \ j \in I.
\end{equation}
Thus,
\begin{equation}\label{Xi2}
\langle p_I, e_j + \sum_i A^i_j e_i \rangle = \lambda_j - \Xi([j] \cap I \cap \{j\})
\quad \forall \ j \in I^c.
\end{equation}

Assume first that $\Delta$ is a Bott polytope. Then
$F_I \cap \Hat{F}_{I^c} \neq \emptyset$ for every non-empty subset $I \subseteq \n$, and so $p_I \in \Delta$. By \eqref{Xi1}, this implies that $\langle p_I, e_{\max(I)} \rangle = \Xi(I) \geq 0$.
In contrast $F_I \cap \Hat{F}_{I^c} \cap \Hat{F}_{\max(I)} = \emptyset$,
and so $\langle p_I , e_{\max(I)} \rangle  = \Xi(I) \neq 0$.
Therefore, $\Xi(I) >0$.

So assume instead that $\Xi(I) > 0$ for every non-empty subset $I \subseteq \n$.
By \eqref{Xi1} and \eqref{Xi2}, this  implies that $\{p_I\} = F_I \cap \hat{F}_{I^c}$
is a vertex of $\Delta$ for every $I \subseteq \n$; 
moreover,  $p_I$ doesn't lie
in $\Hat{F}_j$ for any $j \in I$ or in $F_j$ for any  $j \in I^c$.
Hence,  $F_J \cap \Hat{F}_K \neq \emptyset$ 
for all $J, K \subseteq \n$ such that $J \cap K = \emptyset$.
Moreover, since the vertex $p_I$ lies on exactly $n$ defining hyperplanes,
there are exactly $n$ vertices that are joined to $p_I$ by an edge;
namely, $p_{I \setminus \{i\}}$ for $i \in I$, and $p_{I \cup \{i\}}$ for $i \in I^c$.
Therefore, every vertex of $\Delta$ is equal to $p_I$ for some $I$.
(By Balinski's Theorem, 
the  graph of vertices and edges of any polytope is connected.)
Since every face contains a vertex, this implies that $F_J \cap  \Hat{F}_K = \emptyset$ for
all $J, K \subseteq \n$ such that $J \cap K \neq \emptyset$.
Hence, $\Delta$ is a Bott polytope. \end{proof}

\begin{example}
A strictly upper triangular matrix  $A \in M_3(\Z)$ and $\lambda \in \R^3$
define a symplectic Bott manifold exactly if the
following quantities are positive:
$\lambda_1$, $\lambda_2$, $\lambda_2 -\lambda_1 A_2^1$, 
  $\lambda_3$, $\lambda_3 - \lambda_1 A_3^1$, $\lambda_3 - \lambda_2 A_3^2$, 
and   $\lambda_3 - \lambda_2 A_3^2 - \lambda_1 A_3^1 + \lambda_1 A^1_2 A^2_3 .$
\end{example}

One straightforward way to construct a symplectomorphism between symplectic toric
manifolds is to find an affine transformation between their moment polytopes.
More specifically, let $(M,\omega,\mu)$ and $(\tM,\tomega,\Tilde \mu)$ be $2n$-dimensional symplectic toric manifolds
with  moment polytope $\Delta$ and $\Tilde \Delta$, respectively.  Assume that there
exists a unimodular matrix $\Lambda \in \mat$ and $v \in \R^n$ such that
$\Lambda(\Tilde \Delta) + v = \Delta$. 
Then Delzant's theorem implies that  there exists a symplectomorphism $f$ from
$(\tM,\tomega)$ to $(M, \omega)$
that covers the affine transformation $\Lambda+ v$, i.e.,  the following diagram commutes:
\begin{displaymath}
\xymatrix{
\tM \ar^{f}[r] \ar[d]_{\Tilde \mu}  & M  \ar[d]^{\mu} 
 \\
\Tilde \Delta \ \ar^{\ \ \Lambda+  v\ \ }[r] &\ \Delta \,.
} 
\end{displaymath}
In particular, if $\Lambda + v$ takes a facet $\Tilde G$ of $\Tilde \Delta$ to a facet
$G$ of $\Delta$, then $f$ takes the moment preimage $\Tilde{\mu}^{-1}(\Tilde G)$ to 
the moment preimage $\mu^{-1}(G)$. Therefore, $f^* \colon \hh^*(M;\Z) \to \hh^*(\tM;\Z)$ takes the Poincar\'e dual of
$\mu^{-1}(G)$ to the Poincar\'e dual of $\Tilde{\mu}^{-1}(\Tilde G)$.

\begin{lemma}\label{permute}
Let $(M,\omega)$ and $(\tM,\tomega)$ be the symplectic Bott manifolds associated to strictly upper triangular matrices
$A$ and $\tA$ in $\mat$ and 
$\lambda$ and $\tlambda$  in $\R^n$. 
Let $F \colon \hh^*(M;\Z) \to \hh^*(\tM;\Z)$ 
be a ring isomorphism such that
$F([\omega]) = [\tomega]$ and
$F(x_i) =\tx_{\sigma(i)}$ for all $i$, 
where $\sigma$ is a permutation of $\n$.
Then there exists a symplectomorphism 
$f$ from $(\tM,\tomega)$ to $(M,\omega)$ so that $f^* = F$.
\end{lemma}

\begin{proof}
Let $\Delta$ and $\Tilde \Delta$ be the moment polytopes of $(M,\omega)$ and $(\tM,\tomega)$,
respectively.
Let $\Lambda \in \mat$ be the unimodular matrix taking $e_i$ to $e_{\sigma(i)}$. 
Since $F(x_i) =\tx_{\sigma(i)}$ for all $i$ and 
$F\big( \sum \lambda_i x_i \big) = \sum_i \tlambda_i \tx_i$,
$$ \tA^{\sigma(i)}_{\sigma(j)} = A^i_j \quad \mbox{and} \quad
\tlambda_{\sigma(i)} = \lambda_i \quad \forall \ i,j.$$
Therefore, $\Lambda^T( \Tilde \Delta ) = \Delta$;
moreover,  $\Lambda^T$ takes the facet 
$\{  \langle p, e_{\sigma(j)} + \sum_i \tA_{\sigma(j)}^i e_i \rangle = \tlambda_{\sigma(j)} \} \cap \Tilde \Delta$ to  the facet $\{ \langle p, e_j  + \sum_i A_j^i e_i \rangle = \lambda_j \} \cap \Delta$.
The moment preimages of these facets are Poincar\'e dual to $\tx_{\sigma(i)}$ and $x_i$, respectively.
Therefore, as we saw above, the claim follows from Delzant's theorem.

\end{proof}

\begin{lemma}\label{change sign}
Let $A, \breve A \in \mat$  be strictly  upper-triangular,
and fix $\lambda, \breve \lambda$  in $\R^n$.
Assume that there exists $k \in \n$ so that
\begin{gather*}
\breve A^k_j = -  A^k_j  \quad  \forall \ j,  \\
\textstyle \breve A^i_j  
=  A^i_j - A^i_k  A^k_j \quad \mbox{and} \quad
\textstyle \breve \lambda_j
= \lambda_j -  \lambda_k   A^k_j \quad \forall \  j, \  \forall \  i \neq k.
\end{gather*}
If $(A,\lambda)$ defines a symplectic Bott manifold $(M, \omega)$ 
then 
$(\breve A, \breve \lambda)$ defines a symplectic Bott manifold
$(\breve M, \breve \omega)$; moreover,
there exists a symplectomorphism $f$ from $(M,\omega)$ to $(\breve M,\breve \omega)$
 such that
$f^*$ takes $\breve x_k$ to $x_k + \sum_j A_j^k x_j$
and $\breve x_j$ to $x_j$ for all $j \neq k$.
\end{lemma}

\begin{proof}
Let $\Delta$ be the polytope associated to $(A,\lambda)$ and $\breve{\Delta}$ be the polytope associated to $( \breve{A}, \breve{\lambda})$.
Let $\Lambda \in \mat$  be the  unimodular matrix that takes $e_k$ to 
$-e_k - \sum_i A^i_k e_i$ and takes $e_j$ to itself for all $j \neq k$.  
Then $\Lambda^T(\Delta)  + \lambda_k e_k = \breve{\Delta}$. 
Therefore, if $(A,\lambda)$ defines a symplectic Bott manifold $(M,\omega)$, then
$(\breve A, \breve \lambda)$  defines a symplectic Bott manifold $(\breve M, \breve \omega)$.
Moreover, in this case the affine transformation takes the facet
$\{\langle p, e_k  \rangle = 0 \} \cap \Delta$, whose moment preimage is
Poincar\'e dual to $x_k + \sum_i A_k^i x_i$, to the facet
$\{\langle p, e_k + \sum_i \breve{A}_k^i e_i \rangle = \breve{\lambda}_k \} 
\cap \breve{\Delta}$, whose moment preimage is Poincar\'e dual to  $\breve x_k$.
Similarly,  for all $j \neq k$ the transformation takes 
$\{\langle p, e_j + \sum_i {A}_j^i e_i \rangle = {\lambda}_j \} \cap \Delta$,
whose moment preimage is
Poincar\'e dual to $x_j$, to the facet
$\{\langle p, e_j + \sum_i \breve{A}_j^i e_i \rangle = \breve{\lambda}_j \} 
\cap \breve{\Delta}$, whose moment preimage is Poincar\'e dual to  $\breve x_j$.
Therefore, the claim follows from 
Delzant's theorem.
\end{proof}

Our last application was  used to deduce Corollary \ref{z trivial}.

\begin{proposition}\labell{cube}
Let $\tM$ be a Bott manifold, and let $(M,\omega)$
be a symplectic toric manifold
with $\hh^*(M;\Z)\cong \hh^*(\tM;\Z)$. 
Then  $(M, \omega)$ is symplectomorphic to a symplectic Bott manifold.
\end{proposition}

\begin{proof}
Clearly $M$ and $\tM$ must have the same dimension. Call it $2n$. 
Since  $\tM$ is a Bott manifold it's quotient polytope is an $n$-cube.
Hence, by Theorem 5.5 in \cite{MP},
the fact that $\hh^*(M;\Z_2)=\hh^*(\tM;\Z_2)$ implies that that quotient polytope of $M$ is also an $n$-cube 
(because the cohomology rings with $\Z_2$ coefficients are BQ-algebras in the sense of \cite{MP}).
By Corollary 3.5 in \cite{MP}, this implies that there exists a unimodular matrix in $\mat$
taking the moment polytope of $\Delta$ to a  Bott polytope.
Therefore, the claim follows from Delzant's theorem.
\end{proof}

\section{Constructing symplectomorphisms between Bott manifolds }\label{sec tech}
The main goal of this section is to use
the toric degenerations discussed in Section \ref{sec degenerations}
to construct symplectomorphisms between 
certain symplectic Bott manifolds.
More concretely, 
symplectic Bott manifolds $(M,\omega)$ and $(\tM, \tomega)$
are symplectomorphic
whenever  there exists integers $k, \ell,$ and $\gamma$ with $1 \leq k < \ell \leq n$, 
and an isomorphism from $\hh^*(M)$ to $\hh^*(\tM)$ that takes
$x_k$  to $\tx_k - \gamma \tx_\ell$, $x_i$ to $\tx_i$ for all $i \neq k$,
and $[\omega]$ to $[\tomega]$. 
This result, described in Proposition~\ref{keep}, directly generalizes
the symplectomorphisms between  Hirzebruch surfaces.
In particular,
in general there is no unimodular transformation between their
moment polytopes, so these Bott manifolds are not equivariantly symplectomorphic.
In the next section, 
we will use this result to  prove that, up to symplectomorphism,
every symplectic Bott manifold $(M,\omega)$ has a simplified form. 
To prepare for Proposition~\ref{keep}, we  give a criteria for
such isomorphisms between cohomology rings.

\begin{lemma} \labell{ring} 
Let $M$ and $\tM$ be the  Bott manifolds
associated to distinct strictly upper-triangular matrices $A$ and $\tA$ in $\mat$, 
respectively.
Fix
integers $1 \leq k < \ell \leq n$. The following are equivalent:
\begin{enumerate}
\item There exists $\gamma \in \Z$ and an isomorphism
from $\hh^*(M)$ to $\hh^*(\tM)$ that takes $x_k$ to $\tx_k - \gamma \tx_\ell$ and $x_i$ to $\tx_i$ for all $i \neq k$.
\item
The following equations hold:
\begin{gather*}
\tA^k_\ell = A^k_\ell \pmod 2, \\
\textstyle \tA^i_\ell = A^i_\ell - \frac{1}{2} A^i_k \big(A^k_\ell - \tA^k_\ell\big)  
 \quad  \forall \ i \neq k, \\
\tA^i_j = A_j^i \quad \forall \ i, \  \forall \  j \neq \ell, \\
 \textstyle A^k_j = \frac{1}{2}\big(A^k_\ell +\tA^k_\ell\big) A^\ell_j 
 \quad \forall \ j \neq \ell.
\end{gather*}
\end{enumerate}
Moreover, when (1) is satisfied,  $\gamma = \frac{1}{2} \big( A^k_\ell - \tA^k_\ell \big)$.
Hence,  given $\lambda, \tlambda \in \R^n$, the isomorphism 
takes $\sum_i \lambda_i x_i$
to $\sum_i \tlambda_i \tx_i$  exactly if
$$\textstyle \tlambda_\ell = \lambda_\ell - \frac{1}{2} \lambda_k \big(A^k_\ell - \tA^k_\ell \big) \quad \mbox{and} \quad \tlambda_j = \lambda_j \quad \forall\  j \neq \ell.$$
\end{lemma}

\begin{proof}
Fix  $\gamma \in \Z$, and define a ring homomorphism
$$\Psi \colon \Z[x_1,\dots,x_n] \to \hh^*(\tM;\Z) =\Z[\tx_1,\dots,\tx_n]/
\big(\tx_i^2 + \sum_j \tA^i_j \tx_i \tx_i \big)$$
by $\Psi(x_k) = \tx_k - \gamma \tx_\ell$ and $\Psi(x_i) = \tx_i$ for all $i \neq k$.
Then
\begin{multline*}
\textstyle \Psi\left(x_k^2 + \sum_j A_j^k x_k x_j\right) = \Psi(x_k)^2 + \sum_j A_j^k \Psi(x_k) \Psi(x_j) \\
= \textstyle \left(\tx_k  -  \gamma \tx_\ell \right)^2  + \sum_j  A_j^k \left(\tx_k - \gamma \tx_\ell\right) \tx_j  \\
= \textstyle \tx_k^2  - 2 \gamma \tx_k \tx_\ell  
+ \sum_j  A_j^k \tx_k \tx_j 
+ \left( \gamma^2 - \gamma A_\ell^k \right) \tx_\ell^2   
- \sum_{j \neq \ell} \gamma A_j^k  \tx_\ell \tx_j 
 \\
= \textstyle -\sum_j \tA_j^k \tx_k \tx_j - 2 \gamma \tx_k \tx_\ell  
+ \sum_j A_j^k \tx_k \tx_j
- \sum_j \left( \gamma^2  - \gamma A^k_\ell \right) \tA^\ell_j \tx_\ell \tx_j 
 - \sum_{j \neq \ell} \gamma A_j^k \tx_\ell \tx_j 
\\
= \big(A_\ell^k - \tA_\ell^k - 2\gamma \big) \tx_k \tx_\ell  + \sum_{j \neq \ell} (A_j^k - \tA_j^k) \tx_k \tx_j - \sum_{j \neq \ell} \gamma \left(A_j^k  - A_\ell^k \tA^\ell_j + \gamma \tA^\ell_j\right) \tx_\ell \tx_j.
\end{multline*}
Moreover, for all $i \neq k$, we have
\begin{multline*}
\Psi(x_i^2 + \sum_j A_j^i x_i x_j) = \Psi(x_i)^2 + \sum_j A_j^i \Psi(x_i) \Psi(x_j) \\
= -\sum_j \tA_j^i \tx_i \tx_j  + \sum_{j \neq k} A_j^i \tx_i \tx_j + A_k^i \tx_i (\tx_k - \gamma \tx_\ell) \\
= \sum_{j \neq \ell} (A_j^i - \tA_j^i) \tx_i \tx_j + (A_\ell^i   - \tA_\ell^i - \gamma A_k^i )
\tx_i \tx_\ell.
\end{multline*}
This homomorphism $\Psi$ induces a homomorphism 
$$\hh^*(M)= \Z[x_1,\dots,x_n]/\big( x_i^2 + \sum_j A^i_j x_i x_j \big) \to \hh^*(\tM)$$ exactly if  the right hand sides of these equations vanish, or equivalently, if each term of these expressions vanish.
If $\gamma = 0$, this is  equivalent to $A = \tA$. Otherwise, it
is straightforward to check that these terms vanish exactly if $\gamma = \frac{1}{2} \big( A^k_\ell - \tA^k_\ell \big)$ and the equations in (2) are satisfied.
Finally, it's clear that in this case there's also a ring
homomorphism from $\hh^*(\tM)$ to $\hh^*(M)$ that takes 
$\tx_k$ to $x_k + \gamma x_\ell$ and $\tx_i$ to $x_i$ for 
all $i \neq k$; since this map is the inverse the homomorphism from 
$\hh^*(M)$ to $\hh^*(\tM)$ constructed above, both maps are isomorphisms.
(Alternatively,  the homomorphism $\hh^*(M) \to \hh^*(\tM)$  is a surjective map of free modules of the same dimension.)
\end{proof}

The following  observations will be useful later.

\begin{corollary}\label{c:ring}
Let  $M$ and $\tM$ be Bott manifolds satisfying the equivalent statements of Lemma~\ref{ring}.
Then 
\begin{gather*}
\tA_\ell^i = A_\ell^i \quad \forall \ i > k, \\
\tA_j^k = A_j^k = 0 \quad \forall \ j < \ell. 
\end{gather*}
\end{corollary}

\begin{proof}
The matrix  $A$ is strictly upper triangular. 
Hence, the first claim follows from the facts
that $\tA_\ell^i  = A_\ell^i - \frac{1}{2} A^i_k \big(A^k_\ell - \tA^k_\ell \big)$ for all $i \neq k$, and $A^i_k = 0$ for all $i > k$.
Similarly, the next claim  follows from the facts
$\tA_j^k = A_j^k = \frac{1}{2} \big( A^k_\ell + \tA^k_\ell \big)  A^\ell_j$ 
for all $j \neq \ell$, and $A^\ell_j = 0$  for all $j < \ell$.
\end{proof}
\begin{remark}\label{equiv conditions c}
Let $A \in \mat$ be strictly upper-triangular,
and fix integers  $1 \leq k < \ell \leq n$.
Given an integer $c$, 
there exists a  distinct strictly upper-triangular  matrix $\tA \in \mat$
with $c = \frac{1}{2} \big( A_\ell^k + \tA_\ell^k \big)$
satisfying the statements in Lemma~\ref{ring} 
exactly if  $A_j^k = c A_j^\ell$ for all  $j \neq \ell$. 
In particular, there exist infinitely many such matrices exactly if $A_j^k = A_j^\ell=0$ for 
all $j \neq \ell$; c.f. Definition~\ref{exceptional def}.
\end{remark}
Our next result shows that if there is an isomorphism
of the type considered in Lemma~\ref{ring} between
the cohomology ring  of two Bott manifolds,
and  the ``more skew" polytope is a Bott polytope,
then the ``less skew" polytope is as well.
The converse if not true, as can be seen by considering
trapezoids.
\begin{proposition}\label{old main cor}
Let $M$ and $\tM$ be the Bott manifolds associated to
strictly upper triangular matrices $A$ and  $\tA$ in $\mat$, respectively.
Fix $\lambda, \tlambda \in \R^n$ and integers  $1 \leq k < \ell \leq n$.
Assume that there exists  an isomorphism from $\hh^*(M)$ to $\hh^*(\tM)$
that takes $x_k$ to $\tx_k - \frac{1}{2}\big( A_\ell^k - \tA_\ell^k \big)  \tx_\ell$,  $x_i$ to $\tx_i$ for all
$i \neq k$, and  $\sum \lambda_i x_i$ to $\sum \tlambda_i \tx_i$.
If $|A_\ell^k| \geq |\tA_\ell^k|$
and  $(A,\lambda)$ defines
a symplectic Bott manifold, 
then  $(\tA, \tlambda)$ also defines a symplectic Bott manifold.
\end{proposition}
\begin{proof}
By Proposition~\ref{kahler condition} it suffices to show that $\Xi(I) > 0$ for every non-empty subset $I \subset \n$, assuming that such inequalities hold for $\tXi$. This follows straightforwardly from Lemma \ref{main prop} below and the following observation: if $|A_\ell^k| \geq |\tA_\ell^k|$, then either
$\tA^k_\ell + A^k_\ell \geq 0$ or
$\tA^k_\ell - A^k_\ell \geq 0$; similarly, either
$-\tA^k_\ell + A^k_\ell \geq 0$ or
$-\tA^k_\ell - A^k_\ell \geq 0$.
\end{proof}

\begin{lemma}\label{main prop}
Let $M$ and $\tM$ be the Bott manifolds associated to
strictly upper triangular matrices $A$ and  $\tA$ in $\mat$, respectively.
Fix $\lambda, \tlambda \in \R^n$ and integers $1 \leq k < \ell \leq n$.
Assume that there exists an isomorphism from $\hh^*(M)$ to $\hh^*(\tM)$
that takes $x_k$ to $\tx_k - \frac{1}{2}\big(A_\ell^k - \tA_\ell^k\big) \tx_\ell$,  $x_i$ to $\tx_i$ for all
$i \neq k$, and  $\sum \lambda_i x_i$ to $\sum \tlambda_i \tx_i$.
Let $\Xi$ and $\tXi$ be associated to $(A,\lambda)$ and $(\tA,\tlambda)$, respectively,\footnote{To simplify exposition, we adopt the conventions that $\Xi(\emptyset) = \tXi(\emptyset)$ and $\max (\emptyset)=-\infty$.} 
as in Proposition~\ref{kahler condition}.
Then for any subset $I \subseteq \n \setminus \{k,\ell\}$: 
\begin{itemize}
\item [(a)] 
$\Xi(I) = \tXi(I)$ and $\Xi(I \cup \{k\}) = \tXi(I \cup \{k\}).$
\item [(b)] If  $\max(I) < \ell$, 
let\footnote{
By Claim (a), $\Xi(\k \cap I \cup \{k\}) =   \tXi(\k \cap I \cup \{k\})$.}
 $\Theta:=  \Xi(\k \cap I \cup \{k\})$. Then
\begin{multline*}
\textstyle
\Xi(I \cup \{\ell\}) - \frac{ A_\ell^k \Theta}{2} = \Xi(I \cup \{k,\ell\}) + \frac{A_\ell^k \Theta}{2} \\
= 
\textstyle
\tXi(I \cup \{\ell\}) - \frac{\tA_\ell^k \Theta}{2} = \tXi(I \cup \{k,\ell\}) + \frac{\tA_\ell^k \Theta}{2}.
\end{multline*}
\item [(c)] If  $\max(I) > \ell$, then
$$\Xi(I \cup \{\ell\}) = \tXi(I \cup \{k,\ell\}) \mbox{ and } 
\Xi(I \cup \{k,\ell\}) = \tXi(I \cup \{\ell\}).$$ 
\end{itemize}
\end{lemma}

\begin{proof}
By Lemma~\ref{ring} and Corollary~\ref{c:ring},  all the displayed equations
in both results hold.
Before proceeding further, define\footnote{To simplify exposition,
we adopt the that convention $\Lambda(\emptyset) = \tLambda(\emptyset).$}
$$
\Lambda(J)  := (-1)^{m}  \lambda_{i_0} A^{i_0}_{i_1}  \dots A^{i_{m-1}}_{i_m}$$
for  any non-empty subset  $J = \{i_0,\dots,i_m \} \subseteq \n$
with $i_0 < \dots < i_m$.  By definition, for any non-empty
subset $I \subseteq \n$, 
$$ \Xi(I) = \sum_{\max(I)  \in J \subseteq I} \Lambda(J).$$
Define $\tLambda$ analogously.
Now fix $J \subseteq \n \setminus \{k,\ell\}$.
Then
\begin{equation}\labell{nol}
\Lambda(J) = \tLambda(J)  \mbox{ and } \Lambda(J \cup \{k\}) = \tLambda(J \cup \{k\})
\end{equation}
because $\tA^i_j = A^i_j$ and $\tlambda_j = \lambda_j$ for all $i$ and 
$j \neq \ell$.
Moreover, because also $ \tA_\ell^i = A^i_\ell$ for all $i > k$,
\begin{multline}\labell{equalB}
\Lambda(J \cup \{\ell\}) = \tLambda(J \cup \{\ell\}) \mbox{ and }
\Lambda(J \cup \{k,\ell\}) = \tLambda(J \cup \{k, \ell\})  \\
\mbox{ if } J \cap \{k+1,\dots,\ell -1 \} \neq \emptyset.
\end{multline}
Similarly, the fact that $\tA^k_j = A_j^k = 0$ for all $j < \ell$ implies that
\begin{multline}\label{zero}
 \Lambda(J \cup \{k\} ) =  \tLambda(J \cup \{k\}) =
  \Lambda(J \cup \{k,\ell\} ) =  \tLambda(J \cup \{k,\ell\}) = 0 \\
\mbox{ if }  J \cap \{k+1,\dots,\ell-1\} \neq \emptyset.
\end{multline}
Moreover, by definition, 
\begin{multline}\label{kl}
\Lambda(J \cup \{k,\ell\}) = - A^k_\ell \, \Lambda(J \cup \{k\}) 
\mbox{ and } 
\tLambda(J \cup \{k, \ell\}) = - \tA^k_\ell \, \tLambda(J \cup \{k\})   \\
\mbox{ if }\max(J) < k.
\end{multline}
Next, we claim that
\begin{multline}\label{lkl}
\textstyle
\Lambda(J \cup \{\ell\} ) + \frac{1}{2} \Lambda(J \cup \{k,\ell\}) 
= \tLambda(J\cup \{\ell\}) + \frac{1}{2} \tLambda(J \cup \{k,\ell\}).
\end{multline}
If $J  \cap \{k+1,\dots,\ell-1\} \neq \emptyset$, 
this follows from \eqref{equalB}.
Otherwise,  since $\tA^i_j = A^i_j$ and $\tlambda_j = \lambda_j$ for all $i$ and $j \neq \ell$, depending on whether or not $\ell$ is smaller than every  element of $J$, it
either  follows from the fact that
$\tlambda_\ell = \lambda_\ell- \frac{1}{2} \tlambda_k\big( A^k_\ell - \tA_\ell^k\big)$, or from the fact that
 $\tA^i_\ell = A^i_\ell - \frac{1}{2} \tA^i_k \big(A^k_\ell - \tA_\ell^k \big) $ for all $i \neq k$.
Finally,  we claim that
\begin{multline}\label{kkl}
\textstyle
\Lambda(J \cup \{k\})  = \tLambda(J \cup \{k\}) =
-\frac{1}{2}\Lambda(J \cup \{k, \ell\}) - \frac{1}{2}\tLambda(J \cup \{k, \ell\} )  \\
\mbox{ if }  \max(J)> \ell.
\end{multline}
Indeed, if $J \cap \{k+1,\dots,\ell-1\} \neq \emptyset$, every term of this equation vanishes by \eqref{zero}.
Otherwise, it follows from the facts that
$\tlambda_k = \lambda_k$,  and
$\tA_j^i  = A_j^i$ and
$A_j^k =  \frac{1}{2} \big(A_\ell^k + \tA_\ell^k \big) A_j^\ell$
for all $i$ and $j \neq \ell$, and so
$A_j^k = \tA_j^k = \frac{1}{2} A_\ell^k A^\ell_j = \frac{1}{2} \tA^k_\ell \tA^\ell_j$.

We are now ready to prove the main claims.
So fix a subset $I \subseteq \n \setminus \{k,\ell\}$.

First, claim (a) follows immediately from \eqref{nol}.

Next, assume that  $\max (I)<\ell$.
By \eqref{zero} and \eqref{kl}, 
\begin{multline}\label{xikl1}
\Xi(I \cup \{k,\ell\})  = \sum_{J \subseteq I}\Big( \Lambda(J \cup \{\ell\}) + 
\Lambda(J \cup \{k, \ell\}) \Big) \\
=   \sum_{J \subseteq I} \Lambda(J \cup \{\ell\})  
- \sum_{J \subseteq  \k \cap I} A_\ell^k \Lambda(J \cup \{k\}) 
= \Xi(I \cup \{\ell\} )  - A^k_l \, \Xi(\k \cap I \cup \{k\}).
\end{multline}
Similarly, 
\begin{equation}\label{xikl2}
\tXi(I \cup \{k,\ell\})   
= \tXi(I \cup \{\ell\})  - \tA^k_l \, \tXi(\k \cap I \cup \{k\}).
\end{equation}
Moreover, by \eqref{lkl},
\begin{multline} \label{ximixed}
\Xi(I \cup \{\ell\} ) + \Xi(I \cup \{k,\ell\}) 
= \sum_{J \subseteq I}
\big(2 \Lambda(J \cup \{\ell\}) + \Lambda(J \cup \{k,\ell\}) \big)  \\
= \sum_{J \subseteq I}
\big(2 \tLambda(J \cup\{\ell\}) + \tLambda(J \cup \{k, \ell\})\big)  =
\tXi(I \cup \{\ell\}) + \tXi(I \cup \{k,\ell\}) 
\end{multline}
Claim (b) follows  immediately from
Claim (a) and equations  \eqref{xikl1}, \eqref{xikl2}, and \eqref{ximixed}.

Finally, assume that  $i = \max(I) > \ell$.
By \eqref{nol}, \eqref{lkl}, and  \eqref{kkl}, 
\begin{multline*}
\Xi(I)=  \sum_{i \in J \subseteq I}
\Big(\Lambda(J) + \Lambda(J \cup \{\ell\})\Big) \\
 = \sum_{i \in J \subseteq I } \textstyle \Big( \tLambda(J) + \tLambda(J \cup \{\ell\}) +\frac 1 2 \tLambda(J \cup \{k,\ell\})-\frac 1 2 \Lambda(J \cup \{k,\ell\})\Big)
\\
= \sum_{i \in J \subseteq I }\Big( \tLambda(J) + \tLambda(J \cup \{k\})  
+ \tLambda(J \cup \{\ell\}) + \tLambda(J \cup \{k,\ell\}) \Big) \\
=\tXi(I \cup \{k\}).
\end{multline*}
By a similar argument
$
 \Xi(I \cup \{k\}) 
= \tXi(I);
$
this proves claim $(c)$.
\end{proof}

We are now ready to give a sequence of results culminating in the main result of this section, Proposition \ref{keep}. We begin with the following technical consequence of the above lemma.
\begin{lemma}\label{technical}
Let $M$ and $\tM$ be the Bott manifolds associated to
distinct strictly upper triangular  matrices $A$ and $\tA$  in $M_n(\Z)$, respectively.
Fix $\lambda, \tlambda \in \R^n$ and integers $1 \leq k < \ell \leq n$.
Assume that there exists  
an isomorphism from $\hh^*(M)$ to $\hh^*(\tM)$
that sends $x_k$ to $\tx_k - \frac{1}{2}(A_\ell^k - \tA_\ell^k ) \tx_\ell$,  $x_i$ to $\tx_i$ for all
$i \neq k$, and  $\sum \lambda_i x_i$ to $\sum \tlambda_i \tx_i$.
If $(\tA,\tlambda)$ defines a symplectic Bott manifold, then
\begin{equation}\label{extendposn}
\lambda_\ell - \big\langle p,  \sum_{i \neq k} A_\ell^i e_i \big\rangle 
 > \frac{1}{2}\big( A_\ell^k + \tA_\ell^k \big)
\big( \lambda_k -  \big\langle
p,  \sum_i A^i_k e_i \big \rangle \big) 
\end{equation}
for each  $p \in \R^n$ that satisfies the inequalities
\begin{equation}\label{inequalities}
0 \leq  \langle p, e_j \rangle  \mbox{ and }
\big\langle p, e_j + \sum_i A_j^i e_i \big \rangle \leq \lambda_j \mbox{
for all }j \in \l \setminus \{k,\ell \}.
\end{equation}
\end{lemma}

\begin{proof}
Both sides of the inequality in \eqref{extendposn}
are linear functions on $\R^n$ 
Hence, to show that \eqref{extendposn} holds  for all points in 
the polyhedron defined by the inequalities in \eqref{inequalities},
it is enough to prove it holds on the minimal faces of that polyhedron.
So assume that $p$ lies on such a minimal face.
Then there exists $I \subseteq \l  \setminus \{k,\ell\}$ such that $p \in \R^n$ satisfies
$$\big\langle p, e_j  \big\rangle = \begin{cases}
\lambda_j - \big \langle p, \sum A_j^i e_i \big \rangle &  \mbox{ for all } j \in I  \\
0 & \mbox{ for all other } j \in  \l \setminus \{k,\ell\}.
\end{cases}$$
Since $A$ is strictly upper triangular, $A^i_j = 0$ unless $i < j$. Moreover, 
Corollary~\ref{c:ring} implies that $A^k_j=0$ for all $j < \ell$.
Hence, by induction on $j$,
for all $j \in I$
$$\langle p, e_j \rangle =  
\sum_{\substack{\{i_0,\dots,i_m \} \subseteq I \\ i_0 < \dots < i_m = j }} (-1)^m \lambda_{i_0} A_{i_1}^{i_0} \dots A^{i_{m-1}}_{i_m}=
\Xi(\j \cap I ),$$ 
with the sum taken over subsets $\{i_0,\dots,i_m \} \subseteq  I$ with
$i_0 < \dots < i_m = j$.
Hence,
\begin{gather*}
\lambda_\ell - \big\langle p, \sum_{i \neq k} A_\ell^i  e_i \big\rangle \
= \! \! \! \! \sum_{\substack{ \{i_0,\dots,i_m\} \subseteq I \cup \{\ell\} \\ i_0 < \dots < i_m = \ell}}
\! \! \! \! (-1)^m \lambda_{i_0} A^{i_0}_{i_1} \dots A^{i_{m-1}}_{i_m} 
= \Xi(I \cup \{\ell\}); \\
\lambda_k -  
\big \langle p,  \sum_i A^i_k  e_i \big\rangle  
=  
\! \! \! \! \sum_{\substack{ \{i_0,\dots,i_m\} \subseteq I \cup \{k\} \\ i_0 < \dots < i_m = k}}
\! \! \!  \! (-1)^m  \lambda_{i_0}  A_{i_1}^{i_0} \dots  A_{i_m}^{i_{m-1}}
=
 \,\Xi(\k \cap I  \cup \{k\}).
\end{gather*}
By part (b) of Lemma~\ref{main prop}, 
$$\textstyle \Xi( I \cup \{\ell\}) = \frac{1}{2}(A_\ell^k + \tA_\ell^k)\, \Xi(\k \cap I \cup \{k\}) +  \tXi(I \cup \{k,\ell\}).$$
Finally, by  Proposition~\ref{kahler condition}, if $(\tA,\tlambda)$ defines a symplectic Bott manifold
then $\tXi(I \cup \{k,\ell\})$ is positive.
The claim follows immediately.
\end{proof}

We can now  show how the cohomology ring isomorphisms we are considering here  relate to the ``slide" operator defined
in Definition~\ref{defineF}.
\begin{lemma}\label{l:keep}
Let $(M,\omega)$ and $(\tM,\tomega)$ be the symplectic Bott manifolds
associated to strictly upper triangular matrices $A$ and $\tA$   in $\mat$
and $\lambda$ and $\tlambda$ in $\Z^n$, respectively.
Fix integers $1 \leq k < \ell \leq n$.
Assume that there exists 
an isomorphism from $\hh^*(M)$ to $\hh^*(\tM)$
that takes $x_k$ to $\tx_k - \frac{1}{2}\big( A_\ell^k - \tA_\ell^k \big) \tx_\ell$,
$x_i$ to $\tx_i$ for all
$i \neq k$,
and $\sum \lambda_i x_i$ to $\sum \tlambda_i \tx_i$.
If $\tA_\ell^k \geq |A_\ell^k|$ and 
$c:= \frac{1}{2}\big(A_\ell^k + \tA_\ell^k\big)$, then
\begin{equation}\labell{get}
\mathcal F_{-e_k + c e_\ell}(\Delta(A,\lambda)\cap \Z^n) =  \Delta(\tA,\tlambda) \cap \Z^n.
\end{equation}
\end{lemma}

\begin{proof}
By assumption $c \geq 0$.
By Lemma~\ref{ring}, $c$ is an integer; moreover,
\begin{gather*}\label{keepeq}
\tA^i_\ell = A^i_\ell + A_k^i(c - A^k_\ell) 
 \mbox{ and } \tlambda_\ell = \lambda_\ell + \lambda_k(c - A^k_\ell) 
\quad  \forall \ i \neq k  \\
\tA^i_j = A^i_j \mbox{ and } \tlambda_j = \lambda_j \quad  \forall \  i,\  \forall \ j \neq \ell \\
A_j^k = c A^\ell_j \quad  \forall \   j \neq \ell.
\end{gather*}
Given $p' \in \Z^n$,  consider the affine two-plane
$$V(p') := \{ p \in \R^n \mid \langle p, e_j \rangle = \langle p', e_j \rangle \mbox{ for all }
j \not\in \{k,\ell\} \}.$$
Below, we will fix an arbitrary $p' \in \Z^n$, and prove that
\begin{equation}\labell{need}
\mathcal F_{-e_k + c e_\ell}(\Delta(A,\lambda)\cap  \Z^n  \cap V(p')) = \Delta(\tA,\tlambda)  \cap \Z^n \cap V(p').
\end{equation}
Since the vector $-e_k + c e_\ell$ is parallel to the plane $V$, 
$$\mathcal F_{-e_k + c e_\ell}(\Delta(A,\lambda) \cap \Z^n) \cap V(p') = \mathcal F_{-e_k + c e_\ell} (\Delta(A,\lambda) \cap \Z^n \cap V(p')).$$
Hence, since both equations hold for all $p' \in \Z^n$, we obtain
\begin{equation*}
\mathcal F_{-e_k + c e_\ell}(\Delta(A,\lambda)\cap \Z^n) =  \Delta(\tA,\tlambda) \cap \Z^n,
\end{equation*}
which is exactly the required equation \eqref{get}.

We now fix $p' \in \Z^n$ and turn to proving \eqref{need}. 

Assume first that  $0 > \langle p', e_j \rangle$ for some $j \not\in \{k,\ell\}$.
In this case, $\Delta(A,\lambda) \cap V(p')$ and $\Delta(\tA,\tlambda) \cap V(p')$ are both empty; hence,
\eqref{need} is satisfied.
Similarly, assume that  $\langle p', e_j + \sum_i A^i_j e_i \rangle > \lambda_j$ for
some $j \not\in \{k,\ell\}$ such that $A^\ell_j = 0$.
Then,  $A^k_j = c A^\ell_j = 0$ as well. Therefore, 
$\langle p, e_j + \sum_i A^i_j e_i \rangle = \langle p', e_j + \sum_i A^i_j e_i \rangle > \lambda_j$
for all $p \in V(p')$, and so $\Delta(A,\lambda) \cap V(p')$ is empty.
Since $\tlambda_j = \lambda_j$ and $\tA^i_j = A^i_j$ for all $i$, 
$\Delta(\tA,\tlambda) \cap V(p')$ is also empty.
Once again, \eqref{need} is satisfied.

Therefore, we may assume that $0 \leq \langle p', e_j \rangle$ for all $j \not\in \{k,\ell\}$,
and that 
$\langle p', e_j + \sum_i A^i_j e_i \rangle \leq \lambda_j$ for
all $j \not\in \{k,\ell\}$ such that $A^\ell_j = 0$.
Since these inequalities are also satisfied if we replace $p'$ by any  $p \in V(p')$,
a point $p \in V(p')$ lies in $\Delta(A,\lambda)$ exactly if it satisfies the inequalities
\begin{gather*}
0 \leq \langle p, e_k \rangle \leq \lambda_k - \langle p', \sum_i A^i_k  e_i \rangle, \
0 \leq \langle p, e_\ell \rangle, \\
\langle p, A_\ell^k e_k + e_\ell\rangle \leq \lambda_\ell - \langle p', \sum_{i \neq k} A^i_\ell  e_i \rangle, \ 
\lambda' \leq \langle p, c e_k + e_\ell \rangle \leq  \lambda'',
\end{gather*}
where
\begin{gather*}
\lambda' =
 \max_{\{j \colon  A_j^\ell <  0 \}} \frac{\lambda_j - \langle p', e_j +\sum_{i \not\in\{k,\ell\}} A^i_j e_i \rangle}  {A_j^\ell} ,\\
\lambda'' = 
\min_{\{j \colon  A_j^\ell >  0 \}} \frac{\lambda_j - \langle p', e_j +\sum_{i \not\in\{k,\ell\}} A^i_j e_i \rangle}  {A_j^\ell}.
\end{gather*}
Here, we have used the fact that $A_j^k = c A_j^\ell$ for any $j \neq \ell$,
and adopted the convention that $\min (\emptyset) = \infty$
and $\max (\emptyset) =  - \infty$.
Moreover, since $A_j^\ell = 0$ for all $j < \ell$,
$0 \leq \langle p',e_j \rangle$ and $\langle p', e_j + \sum_i A_j^i e_i \rangle \leq \lambda_j$ for all
$j \in \l \setminus \{k,\ell\}$.
Hence, Lemma~\ref{technical} implies that 
\begin{equation*}
\lambda_\ell - \big\langle p',  \sum_{i \neq k} A_\ell^i e_i \big\rangle 
 > c
\big( \lambda_k -  \big\langle
p',  \sum_i A^i_k e_i \big \rangle \big). 
\end{equation*}
Therefore, by Lemma~\ref{cut trapezoid},
(with $( \lambda_k -  \big\langle
p',  \sum_i A^i_k e_i \big \rangle \big)$ and $(\lambda_\ell - \big\langle p',  \sum_{i \neq k} A_\ell^i e_i \big\rangle)$ playing the roles of $\lambda_1$ and $\lambda_2$, respectively)
an integral point $p \in  \Z^n \cap V(p')$ lies in $\mathcal{F}_{-e_\ell + c e_k}(\Delta(A,\lambda)\cap \Z^n)$ exactly if it satisfies the inequalities
\begin{multline*}
0 \leq \langle p, e_k \rangle \leq \lambda_k - \langle p', \sum_i A^i_k e_i \rangle ,\ 0 \leq \langle p, e_\ell \rangle, \\
\langle p, (2c - A_\ell^k) e_k + e_\ell\rangle \leq 
(c -  A_\ell^k)\big( \lambda_k  - \langle p',  \sum_{i \neq k} A_k^i  e_i \rangle  \big)
+ \lambda_\ell  
 - \langle p',  \sum_{i \neq k} A_\ell^i e_i \rangle , \\
\lambda' \leq \langle p, c \, e_k + e_\ell \rangle \leq  \lambda''.
\end{multline*}
By the equations displayed at the beginning of this proof,
we can rewrite these inequalities as
\begin{multline*}
0 \leq \langle p, e_k \rangle \leq \lambda_k - \langle p', \sum_i \tA^i_k e_i \rangle ,\ 0 \leq \langle p, e_\ell \rangle, \\
\langle p,\tA_\ell^k e_k + e_\ell\rangle \leq 
\tlambda_\ell - 
 \langle p',  \sum_{i \neq k} \tA_\ell^i  e_i \rangle,\   
\lambda' \leq \langle p, c \, e_k + e_\ell \rangle \leq  \lambda''.
\end{multline*}
Finally,
a point $p \in \Z^n \cap V(p')$  lies in $\Delta(\tA,\tlambda)$ exactly if it satisfies the
above inequalities.
Hence, as required, \eqref{need} is satisfied.
\end{proof}

We now have all the tools we need to prove our main proposition.

\begin{proposition}\label{keep}
Let $(M,\omega)$ and $(\tM,\tomega)$ be the symplectic Bott manifolds
associated to strictly upper triangular matrices $A $ and $\tA$  in $\mat$
and $\lambda$ and $\tlambda$  in $\Z^n$, respectively.
Fix integers $1 \leq k < \ell \leq n$.
Assume that there exists 
an isomorphism from $\hh^*(M)$ to $\hh^*(\tM)$
that takes $x_k$ to $\tx_k - \frac{1}{2}\big( A_\ell^k - \tA_\ell^k \big) \tx_\ell$,
$x_i$ to $\tx_i$ for all
$i \neq k$,
and $\sum \lambda_i x_i$ to $\sum \tlambda_i \tx_i$.
Then $(M,\omega)$ and $(\tM,\tomega)$ are symplectomorphic.
\end{proposition}

\begin{proof}
Suppose first that $c:= \frac{1}{2}\big( A_\ell^k + \tA_\ell^k\big) \geq 0$.
By reversing the isomorphism if necessary, we may assume that $\tA_\ell^k \geq A_\ell^k$,
and so $\tA_\ell^k \geq |A_\ell^k|$.
By Lemma~\ref{l:keep},
$$
\mathcal F_{-e_k + c e_\ell}(\Delta(A,\lambda)\cap \Z^n) =  \Delta(\tA,\tlambda) \cap \Z^n.
$$
Moreover, the only properties of $\lambda$ and $\tlambda$ we needed
are the fact that this isomorphism takes $\sum \lambda_i x_i$ to $\sum \tlambda_i \tx_i$,
and the fact that $(A,\lambda)$ and $(\tA,\tlambda)$
define symplectic Bott manifolds.
Therefore, \eqref{get} still holds if we replace $\lambda$ by $m \lambda$ and $\tlambda$ by $m \tlambda$
for some positive integer $m$.
Since $m \Delta(A,\lambda) = \Delta(A,m\lambda)$ and
$m \Delta(A,\tlambda) = \Delta(A,m\tlambda)$, this implies that
$$\mathcal F_{-e_k + c e_\ell}(m \Delta(A,\lambda)\cap \Z^n) =  m \Delta(\tA,\tlambda) \cap \Z^n$$
for all $m \in \Z_{>0}$.
Hence, since $\Delta(A,\lambda)$ and $\Delta(\tA,\tlambda)$ are smooth polytopes,
Proposition~\ref{nice cond give sympl} implies that $(M,\omega)$
and $(\tM,\tomega)$ are symplectomorphic.

So assume instead that  $A_\ell^k + \tA_\ell^k < 0$.
Let $F$ be the isomorphism 
from $\hh^*(M)$ to $\hh^*(\tM)$
that takes $x_k$ to $\tx_k - \frac{1}{2}\big( A_\ell^k - \tA_\ell^k\big) \tx_\ell$,
$x_i$ to $\tx_i$ for all
$i \neq k$,
and $\sum \lambda_i x_i$ to $\sum \tlambda_i \tx_i$.
Define $\breve A, \breve{\Tilde{A}} \in \mat$ and $\breve \lambda, \breve{\Tilde \lambda} \in \R^n$ by
\begin{gather*}
\breve A_j^k = - A_j^k \quad \mbox{and} \quad \breve{\tA_j^k} = - \tA_j^k \quad \forall \ j, \\
\breve A_j^i = A_j^i - A_k^i A_j^k,  \quad 
\breve {\tA_j^i} = \tA_j^i - \tA_k^i \tA_j^k,   \\
 \quad \breve \lambda_j  = \lambda_j - \lambda_k A^k_j, \quad \mbox{and} \quad
  \breve \tlambda_j  = \tlambda_j - \tlambda_k \tA^k_j 
\quad \forall \ j \mbox{ and } \forall \ i \neq k.
\end{gather*}
By Lemma~\ref{change sign}, $(\breve A, \breve \lambda)$ and $(\breve{\tA}, \breve{\tlambda})$
define symplectic Bott manifolds $(\breve M, \breve \omega)$ and $(\breve{\tM}, \breve{\tlambda})$;
moreover, there exist symplectomorphisms $f \colon M \to \breve M$ and $\tilde f \colon \tM \to \breve{\tM}$ 
such that the induced maps in cohomology take $\breve x_k$ to $x_k + \sum_j A_j^k x_j$ and
$\breve \tx_k$ to $\tx_k + \sum_j \tA_j^k \tx_j$, 
and take $\breve x_i$ to $x_i$ and $\breve \tx_i$ to $\tx_i$ for all $i \neq k$.
Moreover, by Lemma~\ref{ring},  $\tA_j^k = A_j^k$ for all $j \neq \ell$.
Therefore, a simple computation shows that
the composition $(\tilde f^*)^{-1} \circ F \circ f^* \colon \hh^*(\breve M)
\to \hh^*(\breve \tM)$ is an isomorphism that takes 
$\breve x_k $ to $\breve \tx_k - \frac{1}{2}\big( \breve {A_\ell^k} - \breve {\tA_\ell^k} \big) \breve \tx_\ell$,
$\breve x_i$ to $\breve \tx_i$ for all $i\neq k$, and
$\sum \breve{\lambda_i} \breve x_i$ to $\sum \breve{\tlambda_i} \breve \tx_i$. 
Since $\breve A_\ell^k + \breve {\tA_\ell^k} > 0$, the claim now follows from 
the first paragraph.
\end{proof}

\begin{remark}
We believe
that the symplectomorphism from $(\tM,\omega)$ to $(M,\omega)$ constructed in Proposition \ref{keep} induces the isomorphism in cohomology described in that proposition.
However, we have not proved this claim here because we didn't need it to prove our main results.
\end{remark}
One consequence of this  proposition is that, given a symplectic Bott manifold,
we can often find a different symplectic Bott manifold  that is symplectomorphic.

\begin{corollary}\label{keycor}
Let $(M,\omega)$ be the symplectic Bott manifold associated to a strictly upper triangular matrix $A \in \mat$ and $\lambda \in \Z^n$.
Assume that there exists $c \in \Z$ and integers $1 \leq k < \ell \leq n$
such that $|2c - A_\ell^k| \leq |A_\ell^k|$ and 
$A_j^k = c A_j^\ell$ for all $j \neq \ell$.
Then $(M,\omega)$ is symplectomorphic to the symplectic Bott
manifold associated to a strictly upper triangular matrix $\tA \in \mat$
and $\tlambda \in \Z^n$  with
\begin{itemize}
\item $\tA_\ell^k = 2c - A_\ell^k$, and otherwise
\item $\tA_j^i = A_j^i$ for all $j$ and all $i \geq k$.
\end{itemize}
\end{corollary}

\begin{proof}
Define a strictly upper triangular matrix $\tA \in \mat$ and $\tlambda \in \Z^n$ by
$\tA_\ell^k =  2c - A_\ell^k$,
$\tlambda_\ell = \lambda_\ell - \lambda_k(A_\ell^k - c)$,
$\tA_\ell^i = A_\ell^i - A_k^i(A_\ell^k -c)$ 
 for all $i \neq k$, and
$\tA_j^i = A_j^i$ and $\tlambda_j = \lambda_j$ for all $i$ and $j \neq \ell$.
Let $\tM$ be the Bott manifold associated to $\tA$.
By  Corollary~\ref{c:ring}, $\tA$ satisfies the equations above;
moreover, by Lemma~\ref{ring},
there's an isomorphism from $\hh^*(M)$ to $\hh^*(\tM)$ that takes
$x_k$ to $\tx_k -(A_\ell^k - c) \tx_\ell$,  $x_i$ to $\tx_i$ for all $i \neq k$,
and $\sum \lambda_i x_i$ to $\sum \tlambda_i \tilde{x}_i$.
Hence, since 
 $|A_\ell^k| \geq  |\tA_\ell^k|$ and
$(A,\lambda)$ defines a symplectic Bott manifold,
Propositions~\ref{old main cor} and \ref{keep} together
imply  that $(\tA,\tlambda)$
defines a symplectic  
Bott manifold that 
is symplectomorphic to $(\tM,\tomega)$.
\end{proof}

\section{Proof of main theorem}\label{sec proof}

The main goal of this section is to prove that strong symplectic cohomological rigidity holds for the family of
$\Q$-trivial symplectic Bott manifolds with rational symplectic forms.
Here, a $2n$-dimensional Bott manifold $M$ is {\bf $\mathbf \Q$-trivial} if its rational cohomology ring 
$\hh^*(M;\Q)$ is isomorphic to $\hh^*((\CP^1)^n;\Q) \simeq \Q[x_1,\dots,x_n]/(x_1^2,\dots,x_n^2)$.
In the process, we obtain a complete classification of $\Q$-trivial symplectic Bott manifolds
with rational symplectic forms up to symplectomorphism; see  Corollary~\ref{cor classification}.

\begin{theorem}\labell{main theorem}
Let $(M,\omega)$ and $(\tM,\tomega)$
be  $\Q$-trivial symplectic Bott manifolds with rational symplectic forms.
Given a ring isomorphism $F \colon \hh^*(M;\Z) \to \hh^*(\tM;\Z)$ such that 
$F\big([\omega]\big) = [\tomega]$,
there exists a symplectomorphism $f$  from $(\tM,\tomega)$
to  $(M,\omega)$ so that $f^*= F$.
\end{theorem}

We begin by constructing some examples of $\Q$-trivial  Bott manifolds.
Given a positive integer $n$, define a strictly upper triangular matrix $A \in \mat$
by  $A_n^i = -1$ for all $i < n$, and  $A^i_j = 0$ for all other $i,j \in [n]$.
Let $\mathcal{H}_n$ be the associated Bott manifold.
By the discussion in the beginning of Section~\ref{s:bm}, 
$$ \mathcal{H}_n = \Big( (\CP^1)^{n-1} \times \big( \C^2 \setminus \{(0,0)\} \big) \Big) /{\C^\times},$$
where $\C^\times$ acts diagonally on $\C^2$, and acts on each $\CP^1$ 
by multiplication on the second coordinate.
Hence, $\mathcal{H}_n$ is a $(\CP^1)^{n-1}$ bundle over $\CP^1$.
More specifically, $\mathcal{H}_2$ is the Hirzebruch surface $\P( \calO(0) \oplus \calO(1)) \cong \P(\calO(-1) \oplus \calO(0))$,
and 
$\mathcal{H}_n$  is the fiber product of $\mathcal{H}_{n-1}$ with $\mathcal{H}_1$
 for all $n \geq 2$. 
By Lemma \ref{cohomology}, the  cohomology of $\mathcal{H}_n$ is given by
$$\hh^*(\mathcal{H}_n;\Z) = \Z[x_1,\dots,x_n]/\big( x_1^2 - x_1x_n, \dots, x_{n-1}^2 - x_{n-1} x_n, x_n^2 \big).$$
Define  $y_i \in \hh^*(\mathcal{H}_n;\Q)$ by 
$y_i = x_i - \frac{1}{2} x_n$ for all $i < n$, and $y_n = x_n$.
The rational cohomology ring
 $\hh^*(\mathcal{H}_n;\Q)$ is isomorphic  to $\Q[y_1,\dots,y_n]/\big(y_1^2, \ldots, y_n^2)$,
as required.

Given $\lambda \in \R^n$, the associated polytope $\Delta(A,\lambda)$
is the set of $p \in \R^n$ such that
$$ \textstyle   \langle p,e_j \rangle \geq 0  \ \forall \  j, \
  \langle p, e_j \rangle \leq \lambda_j \ \forall \ j < n , \mbox{ and }  \langle p, e_n - \sum_{i < n} e_i \rangle \leq \lambda_n.$$
By Proposition~\ref{kahler condition}, $\Delta(A,\lambda)$ is a Bott polytope exactly
if 
$\lambda \in (\R_{>0})^n$.
Assume this holds, and
let $\mathcal{H}_n =\mathcal{H}(\lambda_1,\ldots,\lambda_n)$ be the associated symplectic 
Bott manifold.
More generally, given a partition $\sum_{s=1}^m l_s = n$ of $n$,
and  $\lambda \in (\R_{>0})^n$, define a $\Q$-trivial symplectic Bott manifold
$$\mathcal{H}(\lambda_1, \ldots, \lambda_{l_1})\,\times \dots \times \mathcal{H}(\lambda_{n-l_m+1}, \ldots, \lambda_{n}).$$
Our first lemma, when combined with  Lemma~\ref{permute}, shows
that strong symplectic cohomological rigidity holds for the family of
symplectic Bott manifolds that 
are products of the $\mathcal{H}_i$. 

\begin{lemma}\label{isopermute}
Fix a positive integer $n$.
Let $\sum_{s=1}^m l_s = \sum_{s = 1}^\tm \tl_s = n$
be  partitions  of $n$,
and fix $\lambda, \tlambda \in (\R_{>0})^n$.
Consider the symplectic Bott manifolds
\begin{gather*}
(M,\omega) := \mathcal{H}(\lambda_1,\dots,\lambda_{l_1} )
\times \cdots \times \mathcal{H}(\lambda_{n - l_m + 1},\dots,\lambda_{n}), \\
(\tM,\tomega) := \mathcal{H}(\tlambda_1,\dots,\tlambda_{\tl_1} )
\times \cdots \times \mathcal{H}(\tlambda_{n - \tl_\tm + 1},\dots,\tlambda_{n}).
\end{gather*}
Given a ring isomorphism $F \colon \hh^*(M;\Z) \to \hh^*(\tM;\Z)$ such that $[\omega] = [\tomega]$,
there exists a permutation $\sigma$ of $n$ such that $F(x_i) = \tx_{\sigma(i)}$ for all $i \in [n]$.
\end{lemma}
\begin{proof}
We start with the observation that $F$ must map $\hh^2(M;\Z)$ to $\hh^2(\tM;\Z)$ because the cohomology of a symplectic Bott manifold is generated by elements of degree $2$. 
(In fact, the same is true for any symplectic toric manifold.)
Assume first that both partitions of $n$ are trivial, that is,
$$(M,\omega)= \mathcal{H}(\lambda_1,\dots,\lambda_{n}) 
\quad \mbox{and} \quad 
(\tM,\tomega)= \mathcal{H}(\tlambda_1,\dots,\tlambda_{n}).$$
By inspection,  there are exactly $2n$ primitive
classes in $\hh^2(M;\Z)$ with trivial square: $\pm z_1,\dots, \pm z_n$,
where $z_n = x_n$ and $z_i = 2 x_i - x_n$ for all $i < n$.
Similarly, $\pm \tz_1,\dots, \pm \tz_n$ are the primitive
classes in $\hh^2(\tM;\Z)$ with trivial square,
where $\tz_n = \tx_n$ and
$\tz_i = 2 \tx_i - \tx_n$ for all $i < n$.
Therefore, since $F$ is a ring isomorphism, there exists 
$\epsilon \in \{-1,1\}^i$ and
a permutation $\sigma$ of $\n$ such
that $F(z_i) =  \epsilon_i \tz_{\sigma(i)}$ for all $i$.
(Conversely, this defines an isomorphism for any such $\epsilon$ and $\sigma$.)
Clearly $z_1,\dots,z_n$ is a  basis for $\hh^2(M;\R)$.
In this basis, we can write 
$[\omega] = \sum_{i=1}^n \eta_i z_i$,
where 
$\eta_i = \frac{\lambda_i}{2}$ for all $i < n$
and
$\eta_n = \lambda_n + \sum_{i < n} \frac{\lambda_i}{2}$.
Similarly, $[\tomega] = \sum_{i=1}^n \teta_i \tz_i$, 
where 
$\teta_i = \frac{\tlambda_i}{2}$ for all $i < n$
 and 
$\teta_n = \tlambda_n + \sum_{i< n} \frac{\tlambda_i}{2}$.
Since $F([\omega]) = [\tomega]$, this
implies that $\epsilon_i \eta_i =  \teta_{\sigma(i)}$ for all $i$.
Since $\lambda_i$ and $\tlambda_i$ are positive for all $i$,
$\eta_i$ and $\teta_i$ are positive  for all $i$, and so this implies that $\epsilon_i = 1$  for all $i$.
Moreover, since  $\eta_n > \eta_i$ and $\teta_n > \teta_i$ for all $i < n$,
we have $\sigma(n) = n$.
Thus, $F(x_i) = \tx_{\sigma(i)}$  for all $i$.

We now turn to the general case.
To simplify notation, 
let $i_s = l_1 + \dots + l_s$ for each $s \in \m$, let $i_0 = 0$, and 
let $\lambda^{l_s}$ denote the $l_s$-tuple of
positive numbers $(\lambda_{i_{s-1}+1},\dots, \lambda_{i_s}).$
By inspection, the primitive square zero elements of $\hh^*(M;\Z)$ are precisely
$$\pm x_{i_s} \mbox{ and } \pm(2x_i-x_{i_s}) \textrm{ for }
s \in\m 
\mbox{ and } 
i_{s-1}< i <i_s 
.$$
In particular,
 each primitive square zero element sits in some subring $\hh^*(\mathcal{H}(\lambda^{l_s});\Z)\subseteq \hh^*(M;\Z)$ 
and 
every  primitive square zero element in  $\hh^*(\mathcal{H}(\lambda^{l_s});\Z)$ 
are equal to $x_{i_s}$ modulo $2$. 
Therefore, for each $s \in \m$ there exists $\ts$ with $\tl_\ts = l_s$ 
such that $F$ restricts to an isomorphism from $\hh^*(\mathcal{H}(\lambda^{l_s});\Z)$  to
$\hh^*(\mathcal{H}({\tlambda}^{\tl_\ts});\Z)$ 
and takes $\sum_{i = i_{s-1} + 1}^{i_s} \lambda_i x_i$
to $\sum_{i = \ti_{\ts-1} + 1} ^{\ti_\ts} \tlambda_i \tx_i$.
Therefore, by the previous paragraph, there is a bijection $\sigma_s \colon \{i_{s-1} + 1,\dots, i_s\}
\to \{\ti_{\ts-1} + 1, \dots, \ti_\ts\}$ with $\sigma(i_s) = \ti_{\ts}$
so that  $F(x_i) = \tx_{\sigma_s(i)}$
for all 
$i \in \{i_{s-1} + 1, \dots, i_s \}$.
Combining  $\sigma_1,\dots,\sigma_m$, we construct the required
permutation $\sigma$ of $\n$.
\end{proof}
Our next goal is to prove that
every $\Q$-trivial symplectic Bott manifold is symplectomorphic to
a product of the $\mathcal{H}_i$; see Proposition~\ref{std form}.
Let $M$ be the Bott manifold associated to a strictly upper triangular
matrix $A \in \mat$. Recall that
\begin{equation*}
\hh^*(M;\Z) =\Z[x_1,\dots,x_n]/\big(x_i^2 + \sum_j A^i_j x_i x_j \big).
\end{equation*}
For each $i \in \n$, consider the  special cohomology classes
 \begin{equation}\label{special classes}
 \alpha_i = - \sum_j A^i_j x_j \in  \hh^*(M;\Z) \quad \mbox{and} \quad  y_i = x_i -\frac{1}{2} \alpha_i \in  \hh^*(M;\Q).
 \end{equation}

\begin{definition}
\label{exceptional def}
We say that $x_k$ has {\bf even (odd)
exceptional type} if $\alpha_k = m y_\ell$ for some $\ell \in \n$ and even (respecitvely, odd) integer  $m$.
In ``coordinates", this means that  $A_\ell^k$ is even (respectively, odd)
and that $A_j^k = \frac{1}{2} A^k_\ell A_j^\ell$ for $j \neq \ell$.
Additionally,  we say that $x_k$ has {\bf special} exceptional type if $\alpha_\ell = 0$, that is,
$A_j^k = A_j^\ell = 0$ for all $j \neq \ell$.
\end{definition}

The following lemma,
due to Choi and Masuda \cite[Proposition 3.1]{CM},
  shows that $\Q$-trivial
Bott manifolds have many cohomology classes with exceptional type.
We reprove it here for the reader's convenience.

\begin{lemma}[Choi-Masuda]\label{trivial implies exceptional}
If $M$ is an $2n$-dimensional $\Q$-trivial Bott manifold
then  $x_i \in \hh^*(M;\Z) $ has exceptional type for all $i \in \n$.
\end{lemma}
\begin{proof}
To begin, let  $u \in \hh^2(M;\Z)$  be any  nonzero class with square zero.
Then $u= \sum_j u_jx_j$ for some nonzero $(u_1,\ldots,u_n) \in \Z^n$; let $i \in \n$ be the smallest
number such that 
that $u_i \neq 0$. Then
\begin{align*}
0&=u^2=u_i^2x_i^2 + 2 \sum_{j > i}u_iu_jx_ix_j+ (\textrm{terms involving }x_jx_k\textrm{ for }j,k>i)\\
&=-u_i^2\sum_{j}A^i_jx_ix_j+ 2 \sum_{j \neq i}u_iu_jx_ix_j+ (\textrm{terms involving }x_jx_k\textrm{ for }j,k>i);
\end{align*}
hence,
$ u_i^2A^i_j=2u_iu_j$, and so
$u_j = \frac{1}{2} u_i A^i_j$, for all $j \neq i$.
Therefore,
$u=u_i(x_i -\frac 1 2 \alpha_i) = u_i y_i$ for some $u_i  \in \Z$ and $i \in [n]$.

Since $y_i \in \hh^*(M;\Q)$, there is a unique positive integer $u_i$ such that $u_i y_i$ is a primitive
class in $\hh^*(M;\Z)$ for each $i$.
Therefore, by the previous paragraph, there are twice as many nonzero primitive classes in $\hh^*(M;\Z)$ with square zero as there are $i \in [n]$ such that $y_i^2 = 0$.
On the other hand, since $\hh^*(M;\Q) \simeq \hh((\CP^1)^n;\Q)$,  there exist exactly $2n$ nonzero primitive
elements in $\hh^*(M;\Z)$ with square zero. 
Hence, given any $i \in [n]$,
 $0 = y_i^2 = (x_i - \frac{1}{2}\alpha_i)^2 = \frac{1}{4} \alpha_i^2$.
By the previous paragraph, this implies
that $\alpha_i = m y_\ell$ for some $m \in \Z$ and $\ell \in [n]$.
Hence, $x_i$ has exceptional type.
\end{proof}

Moreover, we can use the symplectomorphisms 
constructed in  Section~\ref{sec tech}
to show that every symplectic Bott manifold
is symplectomorphic to  a Bott manifolds with a simpler cohomology ring,
in the sense that most cohomology classes with exceptional type have
a standard form.

\begin{lemma}\label{keyp}
Let $(M,\omega)$ be a  $2n$-dimensional
symplectic Bott manifold with integral symplectic form.
There exist a strictly upper-triangular matrix $\tA \in \mat$ and $\tlambda \in \Z^n$ so that the associated symplectic Bott manifold $(\tM, \widetilde \omega)$
is symplectomorphic to $(M,\omega)$  and has the following properties for all 
$i \in \n$:
\begin{enumerate}
\item If $\tx_i$ has even exceptional type, then $\Tilde \alpha_i = 0$.
\item 
If $\tx_i$ has odd special exceptional type,
 then $\Tilde \alpha_i = \Tilde y_j$ for some $j > i$.
\end{enumerate}
Here $\Tilde \alpha_i, \Tilde y_i \in \hh(\tM;\Q)=\Q[\tilde {x}_1,\ldots,\tilde {x}_n]/\big(\tilde x_i^2 + \sum_j \tA^i_j \tilde x_i \tilde x_j \big)$ are defined as in \eqref{special classes}.
\end{lemma}

\begin{proof}
By definition, $(M,\omega)$ is the symplectic Bott manifold associated to
some strictly upper-triangular matrix $A \in \mat$ and $\lambda \in \Z^n$.

Since $\alpha_n = 0$, the manifold $M$ has properties (1) and (2) for  $i = n$.
Assume that manifold has properties (1) and (2) for all $i > k$.
We will prove that there exits a 
strictly upper-triangular matrix $\tA \in \mat$ and $\tlambda \in \Z^n$ so that the associated symplectic Bott manifold $(\tM, \widetilde \omega)$
is symplectomorphic to $(M,\omega)$ 
and has properties (1) and (2) for all $i \geq k$.
The claim then follows by induction.

Suppose first that $x_k$ has even exceptional type. 
If $\alpha_k \neq 0$,
then $\alpha_k = -2c y_\ell$ for some nonzero  $c \in \Z$ and $\ell > k$, that is,
$A^k_\ell= 2c$ and   $A^k_j = c A^\ell_j$ for all $j \neq \ell$.
Since $A$ is strictly upper triangular, this implies that $A_j^k = 0$ for all $j < \ell$;
c.f. Corollary~\ref{c:ring}.
Hence, by Corollary~\ref{keycor}, $(M,\omega)$ is symplectomorphic to 
the symplectic Bott manifold $(\tM,\tomega)$ associated to a
strictly upper triangular matrix $\tA \in \mat$ and $\tlambda \in \Z^n$ with
\begin{itemize}
\item  $\tA_j^k  = 0$ for all $j \leq \ell$, and 
\item $\tA^i_j = A_j^i$ for all $i > k$. 
\end{itemize}
The manifold $\tM$ has properties (1) and (2)
for all $i > k$.
Moreover, if $\Tilde x_k$ has even exceptional type
and $\Tilde \alpha_k \neq 0$, then 
$\Tilde \alpha_k = -2c ' \tilde y_{\ell'}$ for some nonzero $c' \in \Z$ and $\ell' > \ell$.
Hence, by repeating this process  as many times as necessary, we can ensure 
that either $\tx_k$ does not have even exceptional type or $\Tilde \alpha_k = 0$, that is,
$\tM$ has property (1) for $i = k$.

So suppose instead that $x_k$ has odd special exceptional type.
Then there exists $c \in \Z$ and $\ell > k$
such that
$A_\ell^k = 2c+1$ and $A_j^k = A_j^\ell = 0$ for all $j \neq \ell$. 
In this case, Corollary~\ref{keycor} implies that
$(M,\omega)$ is symplectomorphic to the symplectic Bott
manifold associated to a  strictly upper 
triangular  matrix $\tA \in \mat$ and $\tlambda \in \Z^n$ with
\begin{itemize}
\item $\tA_\ell^k = -1$ and otherwise $\tA_j^k = 0$ for all $j$, i.e., $\Tilde \alpha_k = \tilde{y}_\ell$, and
\item $\tA^i_j = A_j^i$ for all $i > k$. 
\end{itemize}
Hence, 
$\tM$
has properties  (1) and (2) 
for all $i \geq k$, as required.
\end{proof}

\begin{proposition}\label{std form}
Let $(M,\omega)$ be a $2n$-dimensional $\Q$-trivial symplectic   Bott manifold  with rational symplectic form.
There exists a partition $\sum_{s=1}^m l_s = n$ of $n$ and  $\lambda \in (\Q_{>0})^n$ so that
$(M,\omega)$ is symplectomorphic to  the product
$$\mathcal{H}(\lambda_1, \ldots, \lambda_{l_1})\,\times \cdots \times \mathcal{H}(\lambda_{n-l_m+1}, \ldots, \lambda_{n}).$$
\end{proposition}
\begin{proof}
By rescaling if necessary, we may assume that $\omega$ is integral.

On the one hand, by Proposition~\ref{keyp}, we may assume that $(M,\omega)$
is the symplectic Bott manifold associated to a strictly  upper-triangular matrix $A \in \mat$
and $\lambda \in (\Z_{>0})^n$ with
the following properties for all $i \in \n$:
\begin{enumerate}
\item If $x_i$ has even exceptional type,  then $\alpha_i = 0$.
\item 
If $x_i$ has odd special exceptional type,  then $ \alpha_i =  y_j$ for some $j > i$.
\end{enumerate}

On the other hand, since $M$ is $\Q$-trivial, Lemma~\ref{trivial implies exceptional} below implies
that $x_i$ has exceptional type for all $i$, that is,  $\alpha_i = m y_j$ for some $m \in \Z$ and $j > i$.

So fix any $i\in \n$.  
If $x_i$ has even exceptional type, then (1) implies that $\alpha_i = 0$.
If $x_i$  has odd exceptional type, then $\alpha_i=(2m+1)y_j=(2m+1)(x_j-\frac 1 2 \alpha_j)$ for some
$m \in \Z$ and  $j>i$. 
If $x_j$  also has  odd exceptional type, then 
$\alpha_j=(2m'+1)(x_{j'}-\frac 1 2 \alpha_{j'})$ for some $m' \in \Z$
and $j' > j$, and so $A^i_{j'}=\frac 1 2(2m+1) (2m'+1) \notin \Z$. Since this is impossible,
$x_j$ has even exceptional type. As we've seen, this implies that $\alpha_j = 0$,
that is, $x_i$ has special exceptional type.
By  (2) this implies that  $\alpha_i=y_j$.
In conclusion, either $\alpha_i = 0$, or there exists $j > i$ such that $\alpha_i = y_j$ and 
$\alpha_j = 0$.

Let $\mathcal{E} = \{i_1,\dots,i_m\}$ be the set of $i \in \n$  such that $\alpha_i = 0$.
For each $i \in \mathcal{E}$  let
$$B_i:=\{i\} \cup \{ j \mid  \alpha_j= y_i\}$$
and let $l_i$ be the number of elements in $B_i$.
By the previous paragraph,  $B_1,\dots, B_m$  is  a partition of $\n$,
that is $l_1 + \dots + l_m = n$.
Hence, there exists a permutation $\sigma$ of $\n$ such that, for all $j \in \n$ and $s \in \m$,
we have $\sigma(i_{s-1}) < \sigma(j) \leq \sigma(i_s)$ exactly if $j \in B_{i_s}$. (Here, we let $\sigma(i_0) = 0$.)
Choose $\tlambda \in (\Q_{>0})^n$ so that $\lambda_i = \tlambda_{\sigma(i)}$
for all $i$, and consider the symplectic Bott manifold
$$(\tM, \tomega) = 
\mathcal{H}(\tlambda_1, \ldots, \tlambda_{l_1})\,\times \cdots \times \mathcal{H}(\tlambda_{n-l_m+1}, \ldots, \tlambda_{n}).$$
There is a ring isomorphism $F \colon \hh^*(M) \to \hh^*(\tM)$ 
such that $F([\omega]) = [\tomega]$ and $F(x_i) = \tx_{\sigma(i)}$ for all $i$.
Therefore, the claim follows from Lemma~\ref{permute}.
\end{proof}

The proof of the main theorem is now immediate.
\begin{proof}[Proof of Theorem \ref{main theorem}]
Let $(M,\omega)$ and $(\tM,\tomega)$
be  $\Q$-trivial symplectic Bott manifolds with rational symplectic forms,
and let $F \colon \hh^*(M;\Z) \to \hh^*(\tM;\Z)$ be a ring isomorphism such that 
$F\big([\omega]\big) = [\tomega]$. 
Let $n = \dim \hh^2(M;\Q)$.
By assumption,
$$\dim M =  2 \dim \hh^2(M;\Q) = 2 \dim \hh^2(M;\Q) = \dim \tM.$$
Hence, by Proposition \ref{std form} there exist partitions
 $\sum_{s=1}^m l_s = \sum_{s = 1}^\tm \tl_s = n$ of $n$
and $\lambda, \tlambda \in (\Q_{>0})^n$ such that
\begin{gather*}
(M,\omega) := \mathcal{H}(\lambda_1,\dots,\lambda_{l_1} )
\times \cdots \times \mathcal{H}(\lambda_{n - l_m + 1},\dots,\lambda_{n}), \\
(\tM,\tomega) := \mathcal{H}(\tlambda_1,\dots,\tlambda_{\tl_1} )
\times \cdots \times \mathcal{H}(\tlambda_{n - \tl_\tm + 1},\dots,\tlambda_{n}).
\end{gather*}
Hence, combining Lemmas~\ref{isopermute} and \ref{permute},
there exists a symplectomorphism $f$  from  $(\tM,\tomega)$ to
$(M,\omega)$ so that that $f^* = F$.
\end{proof}
As a corollary, we obtain a complete classification up to symplectomorphism  of $\Q$-trivial symplectic Bott manifolds with rational symplectic forms.
Recall that  a {\bf pointed set} is  a pair $(A,a)$, where $A$ is a set
and $a \in A$.
\begin{corollary}\label{cor classification}
Fix a positive integer $n$.
There is a one-to-one correspondence between
symplectomorphism classes of $2n$-dimensional symplectic Bott manifolds with rational symplectic forms
and multisets of $n$ positive rational numbers  equipped with a partition into pointed multisets.
\end{corollary}
\begin{proof}
As we showed in the beginning of this section, there is a  symplectic Bott manifold
with rational symplectic form 
$$(M,\omega)
 = \mathcal{H}(\lambda_1, \ldots, \lambda_{l_1})\,\times \dots \times \mathcal{H}(\lambda_{n-l_m+1}, \ldots, \lambda_{n})$$
associated to any partition $\sum_{i=1}^m l_i = n$ of $n$ and $\lambda \in (\Q_{>0})^n$.

If $(\tM,\tomega)$ is the symplectic Bott manifold
associated to another
partition $\sum_{i=1}^\tm \tl_i = n$ of $n$ and $\tlambda \in (\Q_{>0})^n$, 
then $M$ and $\tM$ are symplectomorphic exactly if $\{\lambda_i\}$ and $\{\tlambda_i\}$ agree as multisets with partitions into pointed multisets.
To see this, assume first that  $M$ and $\tM$ are symplectomorphic.
Then there exists a ring isomorphism  $F \colon \hh^*(M;\Z) \to \hh^*(\tM;\Z)$ such that $F([\omega]) = [\tomega]$.
Hence, by Lemma~\ref{isopermute}, 
there exists a permutation $\sigma$ of $[n]$ such that $F(x_i) = \tx_{\sigma(i)}$ for all $i \in [n]$.  This implies that $\lambda_i = \tlambda_{\sigma(i)}$
for all $i$, and that $A_j^i = \tA_{\sigma(j)}^{\sigma(i)}$ for all $i,j \in [n]$, or equivalently, that  $\sigma$ induces an isomorphism of multisets with partitions into pointed multisets. 
Conversely, if a permutation $\sigma$ of $[n]$ induces an isomorphism of
multisets with partitions into pointed multisets, then it also induces
a ring isomorphism $F \colon \hh^*(M;\Z) \to  \hh^*(\tM;\Z)$ 
such that $F([\omega]) = [\tomega]$ and $F(x_i) = F(\tx_{\sigma(i)})$ for all $i$.
Therefore, by Lemma~\ref{permute}, $M$ and $\tM$ are symplectomorphic.

Finally, by Proposition \ref{std form}, every symplectic Bott manifold with rational symplectic
form is symplectomorphic to the  symplectic Bott manifold  associated
to some partition $\sum_{i=1}^m l_i = n$ and some $\lambda \in (\Q_{>0})^n$.
\end{proof}

In comparison, by a theorem of Choi and Masuda,
the diffeomorphism type of $M$ is determined by the partition of $n$
\cite[Theorem 4.1]{CM}.

\end{document}